\documentclass[11pt]{amsart}
\usepackage[foot]{amsaddr}
\usepackage[utf8]{inputenc}
\usepackage[T1]{fontenc}
\usepackage{lmodern}
\usepackage[pdftex]{graphicx}
\usepackage{amssymb}
\usepackage{amsmath}
\usepackage{amsfonts}
\usepackage{amsthm}
\usepackage{bm}
\usepackage{mathrsfs}
\usepackage{appendix}
\usepackage{enumerate}
\usepackage[left=2.5cm, right=2.5cm, bottom=3cm]{geometry}
\usepackage[stretch=30,shrink=30]{microtype}
\IfFileExists{stmaryrd.sty}{\usepackage{stmaryrd}}{}
\IfFileExists{ulem.sty}{\usepackage[normalem]{ulem}}{}

\usepackage{xcolor}
\definecolor{antiquefuchsia}{rgb}{0.57, 0.36, 0.51}
\definecolor{azure}{rgb}{0.0, 0.5, 1.0}

\usepackage[backref=page, colorlinks = true, linkcolor = black, citecolor = antiquefuchsia]{hyperref}

\renewcommand*{\backref}[1]{}
\renewcommand*{\backrefalt}[4]{%
    \ifcase #1 (Not cited.)%
    \or        (Cited on page~#2.)%
    \else      (Cited on pages~#2.)%
    \fi}

\makeatletter
\def\th@plain{%
	\thm@notefont{}
	\itshape 
}
\def\th@definition{%
	\thm@notefont{}
	\normalfont 
}
\makeatother

\newcommand\res{\mathop{\hbox{\vrule height 7pt width .3pt depth 0pt\vrule height .3pt width 5pt depth 0pt}}\nolimits}

\numberwithin{equation}{section}

\newtheorem{theorem}{Theorem}[section]
\newtheorem{lemma}[theorem]{Lemma}
\newtheorem{proposition}[theorem]{Proposition}
\newtheorem{corollary}[theorem]{Corollary}

\theoremstyle{definition}
\newtheorem{definition}[theorem]{Definition}

\theoremstyle{remark}
\newtheorem{remark}[theorem]{Remark}
\newtheorem{example}[theorem]{Example}

\newcommand{\N}{\mathbb{N}}

\newcommand{\R}{\mathbb{R}}

\DeclareMathOperator{\vol}{vol}

\DeclareMathOperator{\loc}{loc}

\DeclareMathOperator{\YM}{YM}

\usepackage[capitalise,nameinlink]{cleveref}
\crefformat{equation}{#2(#1)#3}
\crefrangeformat{equation}{#3(#1)#4 to~#5(#2)#6}
\crefmultiformat{equation}{#2(#1)#3}{ and~#2(#1)#3}{, #2(#1)#3}{ and~#2(#1)#3}

\newcommand{\g}{\mathfrak{g}}

\newcommand{\eps}{\varepsilon}

\title[Uniqueness of tangents and log-epiperimetric inequality]{Almost minimizing Yang--Mills fields: log-epiperimetric inequality, non-concentration, and uniqueness of tangents}

\author[Riccardo Caniato]{Riccardo Caniato}
\address{ \mbox{\it{(R. Caniato)} } Mathematics Institute, University of Warwick, Coventry CV4 7AL, United Kingdom}
\email{riccardo.caniato@warwick.ac.uk}

\author[Davide Parise]{Davide Parise}
\address{ \mbox{\it{(D. Parise)} } Mathematics Institute, University of Warwick, Coventry CV4 7AL, United Kingdom}
\email{Davide.Parise@warwick.ac.uk}

\date{\today} 

\hypersetup{pdftitle={Almost minimizing Yang--Mills fields: log-epiperimetric inequality, non-concentration, and uniqueness of tangents},pdfauthor={Riccardo Caniato and Davide Parise}}
\begin{document}
 
\begin{abstract}
    We establish a direct log-epiperimetric inequality for Yang--Mills fields in arbitrary dimension and we leverage it to prove uniqueness of tangent cones with isolated singularity for almost energy minimizing connections and $\omega$-ASD connections (where $\omega$ is not necessarily closed) satisfying some suitable regularity and non-concentration assumptions. En route to this we establish a Luckhaus type lemma for connections to exclude curvature concentration along blow-up sequences. 
\end{abstract}

\maketitle

\begin{center}
\end{center}

\tableofcontents

\section{Introduction}
\subsection{The general setting: Yang--Mills fields}
The aim of this article is to understand the behavior of Yang--Mills connections at their singular points, and prove uniqueness of the corresponding tangent cones whenever they happen to satisfy certain structural properties. This fits broadly into a line of investigation aiming at understanding the regularity of extrema of geometric variational problems, and solutions of partial differential equations of geometric type. 

Let us recall the framework we will be working in. Let $G$ be a compact matrix Lie group with Lie algebra $\mathfrak{g}$, equipped with an $\operatorname{Ad}$-invariant inner product, and let $(N,h)$ be a smooth compact $n$-dimensional Riemannian manifold with $n\ge 2$, possibly with smooth boundary $\partial N$. The \textit{Yang--Mills functional} on a principal $G$-bundle $P$ over $N$ is given by 
\begin{align}\label{equation: Yang-Mills intro}
        \YM_N(A):=\int_{N}\lvert F_A\rvert^2\, d\vol_h \qquad\forall\, A\in A_0+(W^{1,2}\cap L^4)(N,T^*N\otimes\g_P), 
\end{align}
where $A_0$ is any smooth reference connection on $P$, the 2-form
\begin{align*}
    F_A:=F_{A_0}+d_{A_0}(A-A_0)+(A-A_0)\wedge(A-A_0)\in L^2(N,\wedge^2T^*N\otimes\g_P)
\end{align*} 
is the curvature of the principal $G$-connection $A$, and the norm $\lvert F_A\rvert$ is computed with respect to the $\operatorname{Ad}$-invariant inner product on the \textit{adjoint bundle} $\g_P$\footnote{The adjoint bundle $\g_P$ is the vector bundle $\tilde\pi:\mathfrak{g}_P\to N$ over $N$ whose total space is given by
    \begin{align*}
        \mathfrak{g}_P:=P\times_{\operatorname{Ad}}\mathfrak{g}=\frac{P\times\mathfrak{g}}{\sim_{\operatorname{Ad}}}
    \end{align*}
    where
    \begin{align*}
        (p_1,v_1)\sim_{\operatorname{Ad}}(p_2,v_2) \quad\Leftrightarrow\quad \exists\,g\in G \mbox{ : } p_2=p_1g,\,v_2=g^{-1}v_1g
    \end{align*}
    and whose projection on the base manifold $N$ is defined to be $\tilde\pi([(p,v)]):=\pi(p)$ for every $[(p,v)]\in\mathfrak{g}_P$.}
over $N$. We realize $G$ as a unitary matrix group and use the normalization $\langle X,Y\rangle=-\operatorname{tr}(XY)$. 
Critical points of $\YM_N$ are called \textit{Yang--Mills fields} or \textit{Yang--Mills connections} on $P$.\footnote{If $P:=N\times G$ is the trivial principal $G$-bundle over $N$, for short, we say that critical points of $\operatorname{YM}_N$ are Yang--Mills fields on $N$.} Geometrically, the Yang--Mills energy measures how far a certain connection is from being flat in the $L^2$-sense. From an analytic perspective, $\YM_N$ is a conformally invariant\footnote{Meaning invariant under conformal changes of the domain metric.} lagrangian in its 
\textit{critical} dimension $n=4$, whose properties make the associated variational problems rich and challenging.  
For instance, the functional is \textit{gauge invariant} in the following precise sense: for every \textit{local gauge transformation} $g\in(W^{2,2}\cap W^{1,4})(N,G_P)$\footnote{Here $G_P$ stands for the \textit{conjugated bundle}, i.e. the $G$-bundle over $N$ whose total space is given by
    \begin{align*}
        G_P:=P\times_{c}G=\frac{P\times G}{\sim_c}
    \end{align*}
    where
    \begin{align*}
        (p_1,h_1)\sim_{c}(p_2,h_2) \quad\Leftrightarrow\quad \exists\,g\in G \mbox{ : } p_2=p_1g,\,h_2=g^{-1}h_1g
    \end{align*}
    and whose projection on the base manifold $N$ is defined to be $\hat\pi([(p,h)]):=\pi(p)$ for every $[(p,h)]\in G_P$.} 
and for every connection $A$, the \textit{gauge transformed} connection, written in local representatives,
\begin{align*}
    A^g:=g^{-1}dg+g^{-1}Ag\in A_0+(W^{1,2}\cap L^4)(N,T^*N\otimes\g_P)
\end{align*}
satisfies
\begin{align*}
    \YM_N(A^g)=\YM_N(A).
\end{align*}
The presence of such a large invariance group for $\mathrm{YM}_N$ leads to several remarkable analytical consequences. For example, the Euler--Lagrange equations, called \textit{Yang--Mills equations}, and the stability operator associated with $\YM_N$ are \textit{not elliptic}, at least until a proper \textit{relative Coulomb gauge} is chosen, cf. Remark \ref{Remark: ellipticity of the second variation of the YM-energy discrepancy}. Besides, the nonlinearity $A\wedge A$ appearing in these equations is unwieldy, as it drastically breaks the coercivity of the functional leading to concentration-compactness phenomena in dimension $n\ge 4$.\footnote{See e.g. \cite{uhlenbeck-connections} and \cite{tian-gauge-theory}, where the authors studied the concentration-compactness phenomena for the Yang--Mills functional respectively in critical and supercritical dimension.}

The Yang--Mills functional has played a major role in the understanding of the differential geometry of 4-manifolds. Let us recall a few key examples, without aiming for completeness (see e.g. \cite{donaldson-kronheimer} or \cite{freed-uhlenbeck} for a detailed discussion of the subject). Donaldson proved his celebrated result on the existence of non-smoothable topological $4$-manifolds by studying properties of the moduli space of instantons (self-dual or anti-self-dual solutions of the Yang--Mills equations) over smooth four-manifolds (see \cite{Donaldson}). Notably, these manifolds had already been constructed by Freedman a year earlier in his solution to the topological 4-dimensional Poincaré conjecture \cite{Freedman}. A second remarkable application of gauge theory in this context is the proof of existence of exotic differentiable structures on $\R^4$. The existence of at least one ``fake'' $\R^4$ stems largely from the aforementioned works of Donaldson and Freedman,  with Gompf later providing an explicit construction and showing the existence of at least three such exotic structures \cite{Gompf}. Finally, Taubes in \cite{Taubes} achieved the sharpest result in this sense, showing the existence of uncountably many fake $\R^4$'s by means of gauge theoretic methods. 

Given the effective applications of $\YM_N$ in four dimensions, it is natural to investigate its behaviour in higher dimensions, i.e. in the \textit{supercritical} regime. Notably, a research program along these lines was proposed by Donaldson and Thomas in \cite{DonaldsonThomas}. Unfortunately, the analysis of the Yang--Mills lagrangian becomes more challenging in dimension greater than four and we are confronted with the study of \textit{singular} solutions, as they naturally arise in this wilder framework. As in other geometric variational problems, it is useful to consider tangent cones, \textit{tangent connections} in this case. These are weak limits of rescalings of the original connection, capturing the local behaviour around a given point. As these limiting objects are supposed to model the behaviour around a (singular) point, understanding whether they are unique is an important issue. A priori, at a given singular point the connection may asymptotically approach one cone at certain scales, and a different cone at others.

A method to prove uniqueness of tangents with \textit{isolated singularity} was pioneered by Simon in \cite{AsymptoticSimon} for stationary varifolds, and harmonic maps. The idea is to prove an infinite dimensional version of the classical \L ojasiewicz gradient inequality for analytic functions in Euclidean space, cf. Lemma \ref{lemma: finite dimensional Lojasiewicz inequality}. In gauge theory, this method was first introduced by Morgan, Mrowka, and Ruberman in their work \cite{MMR} on the Chern--Simons functional, while R\aa de \cite{Rade} later adapted Simon's proof in dimensions 2 and 3 for the Yang--Mills functional. It was only later that Yang \cite{Yang} used Simon's method in any dimension to prove that, under strong curvature bounds, tangent cones at isolated singularities are unique. In this article we approach this question through a direct log-epiperimetric inequality, treating almost minimizers and $\omega$-ASD connections. 
We also refer the reader to more recent works of Feehan, partly in collaboration with Maridakis, establishing various gradient \L ojasiewicz inequalities \cite{FeehanMemoir, Feehan, FeehanMaridakis, FeehanMaridakisCrelle}. 

\L ojasiewicz--Simon inequality type arguments are not directly applicable if the singularity of the cone is not isolated, as they usually require the cone to have a smooth link. In the case of pseudoholomorphic maps and semicalibrated currents, slicing techniques proved to be extremely effective (see e.g. \cite{pumberger-riviere}, \cite{riviere-tian-maps}, \cite{caniato-riviere}). In the particular case of integral $p-p$ cycles, we also mention the work \cite{bellettini-p} by Bellettini, in which the author develops the so-called \textit{algebraic blow-up method}. We currently conjecture that these types of techniques could be the key to tackle the uniqueness of tangents for general $\omega$-ASD connections\footnote{See Definition \ref{Definition: omega-ASD connections}.} on general almost complex manifolds. Note that in the setting of stationary integral $1$-varifolds, uniqueness of tangents is completely settled by \cite{AllardAlmgren}. We refer the reader to the surveys \cite{de2022regularity, wickramasekera2014regularity} for further details on the uniqueness of tangent cones problem in the setting of minimal submanifolds. 
\subsection{Statement of the main results}
In comparison with previous literature on the subject, we make minimal regularity assumptions on the class of connections under consideration. Secondly, we go beyond the K\"ahler setting, in which the richer complex structure, and in particular the existence of local holomorphic coordinates, simplifies the analysis and allows for finer considerations. See below for further details on this aspect. Besides, our epiperimetric argument is variational, and does not exploit any underlying PDE, thus allowing us to treat in a unified and more systematic way a larger class of extrema, i.e. almost minimizers. For $\rho>0$, we write $\tau_{y,\rho}(x):=\rho x+y$ and $A_{y,\rho}:=\tau_{y,\rho}^*A$. Our first main theorem is the following.

\begin{theorem}[Uniqueness of tangents]\label{Theorem: uniqueness of tangent connections with isolated singularities}
    Let $G$ be a compact matrix Lie group with Lie algebra $\mathfrak{g}$. Let $n\ge 5$, let $\Omega\subset\R^n$ be an open set, let $P\to\Omega$ be a smooth principal $G$-bundle, and let $y\in\Omega$. Assume that the following facts hold.
    \begin{enumerate}
        \item $A$ is a connection on $P$ with $W^{1,2}_{\loc}\cap L^4_{\loc}$ local connection forms, and is either $\omega$-ASD for a smooth form satisfying the algebraic bound \eqref{Equation: algebraic bound for omega} at $y$, or a $\operatorname{YM}$-energy minimizer\footnote{See Definition \ref{Definition: almost YM-energy minimizers}.}; in either case, assume that $\mathscr{H}^{n-4}(\operatorname{Sing}(A)\cap K)<+\infty$ for every compact set $K\subset\Omega$.\footnote{Here and in what follows, we let 
        \begin{align*}
            \operatorname{Reg}(A):=\{x\in\Omega \mbox{ : } A\mbox{ admits a smooth local representative modulo gauge on }B_{\rho}(x)\mbox{ for some }\rho>0\}
        \end{align*}
        and $\operatorname{Sing}(A)=\Omega\smallsetminus\operatorname{Reg}(A)$. Moreover, $\mathscr{H}^{k}$ denotes the $k$-dimensional Hausdorff measure on $\R^n$.}
        \item $y\in\operatorname{Sing}(A)$ and $\varphi$ is a tangent cone for $A$ at $y$ such that $\operatorname{Sing}(\varphi)=\{0\}$, its restriction to $\mathbb S^{n-1}$ admits a smooth global representative on the trivial bundle, and $F_{A_{y,\rho_i}}\to F_{\varphi}$ strongly in $L^2(\mathbb B_2)$ (modulo gauge transformations) along some sequence of rescalings $\rho_i\to 0$ as $i\to+\infty$.
    \end{enumerate}
    \noindent
    Then, $\varphi$ is the unique tangent cone to $A$ at $y$ (modulo gauge transformations). Moreover, the decay is logarithmic, i.e. there exist $\alpha>0$ and constants $C_k>0$ for every $k\in\N$ such that
    \begin{equation*}
       \vert \tau_{y,\rho}^*A - \varphi \vert_{C^k(\mathbb{S}^{n - 1})} \leq C_k \vert \log(\rho) \vert^{-\alpha}, 
    \end{equation*}
    where the estimates are understood after suitable gauge transformations
    on $\mathbb S^{n-1}$. Moreover, if we assume $\varphi$ to be integrable, the rate of convergence improves to 
    \begin{equation*}
        \vert \tau_{y,\rho}^*A - \varphi \vert_{C^k(\mathbb{S}^{n - 1})} \leq C_k\rho^{\alpha}.     
    \end{equation*}
\end{theorem}

Under the algebraic bound at $y$, an $\omega$-ASD connection is almost minimizing for comparisons on balls centered at $y$; see Remark \ref{Remark: centered almost minimality}.
\noindent
We briefly comment on the main hypotheses in the above theorem. 
\begin{itemize}
\item $\mathscr{H}^{n-4}(\operatorname{Sing}(A)\cap K)<+\infty$ for every compact set $K\subset\Omega$. Together with smoothness modulo gauge on the regular set, this makes $A$ an \textit{admissible connection} in the sense of \cite{tian-gauge-theory}. Thus, an $\eps$-regularity statement is available for $A$ by \cite{TaoTian} for minimizers and \cite{wentworth-chen} for smooth $\omega$-ASD equations. Note that an $\eps$-regularity theorem under a strong approximability hypothesis was established by Meyer and Rivière in \cite{meyer-riviere}. For related regularity results for stationary \textit{weak $L^2$-connections} in dimension $5$, in the framework of \cite{petrache-riviere}, see \cite{caniato-riviere-ym}.
\item The tangent cone $\varphi$ has an isolated singularity at the origin. This is the usual hypothesis appearing in \cite{AsymptoticSimon, EngelsteinSpolaorVelichkov, ESVduke} and beyond which it is incredibly difficult to go. In the setting of mean curvature flow, Colding and Minicozzi have been able to deal with (generic) \textit{cylindrical} singularities, see \cite{ColdingMinicozzi}. We also refer the reader to work of Székelyhidi on cylindrical tangent cones, \cite{Gabor}, and to earlier work of Simon \cite{CylindricalJDG, CylindricalCAG}. 
\item Non-concentration of the measures, i.e. emptiness of the bubbling locus. This is another additional hypothesis due to the potential bubbling of sequences of Yang--Mills connections. In particular, the curvatures of the rescaled connections $\tau_{y, r}^* A$ could exhibit a concentration set larger than the actual singular set of $\varphi$. This possibility is ruled out in the context of harmonic maps. Indeed, when considering a sequence of \textit{energy minimizing} harmonic maps weakly converging in $W^{1, 2}$, we know that the convergence is in fact strong \cite{schoen-uhlenbeck}, and the limit is energy minimizing as well \cite{HardtLin, Luckhaus}.   
\end{itemize}
We address this last difficulty by proving an annular interpolation estimate and a non-concentration criterion in the natural class $W^{1,2}\cap L^4$ (cf. Lemma \ref{Lemma: no concentration}). For interpolation in the larger class of weak connections in dimension five, see \cite{petrache-riviere}; for concentration in the parabolic Yang--Mills flow, see \cite{Waldron}.

In every dimension $n\ge5$, the following criteria exclude curvature concentration in the setting of Theorem \ref{Theorem: uniqueness of tangent connections with isolated singularities}.
\begin{corollary}[Uniqueness of tangents excluding concentration]\label{Theorem: higher dimensions}
Let $A$, $y$ and $\varphi$ satisfy the assumptions of Theorem \ref{Theorem: uniqueness of tangent connections with isolated singularities}, with weak convergence of the curvatures in $L^2(\mathbb B_2)$ in place of strong convergence in assumption (2). In a smooth trivialization near $y$, write $A_i:=\tau_{y,\rho_i}^*A$ for the corresponding rescalings and set $U:=\mathbb B_{3/2}\smallsetminus\overline{\mathbb B_1}$. Suppose that at least one of the following conditions holds:
\begin{enumerate}[(i)]
\item There are gauges $g_i\in W^{2,2}(U,G)\cap W^{1,4}(U,G)$ such that, writing $B_i:=A_i^{g_i}$,
\begin{align*}
\|B_i-\varphi\|_{L^2(U)}^2
\left(1+\|(B_i-\varphi)\wedge(B_i-\varphi)\|_{L^2(U)}^2\right)\rightarrow0.
\end{align*}
\item $|F_{A_i}|\to|F_\varphi|$ almost everywhere on $U$, and the functions $|F_{A_i}|^2$ are uniformly integrable on $U$.
\item $A$ is $\omega$-ASD and there are gauges $g_i\in W^{2,2}(U,G)\cap W^{1,4}(U,G)$ such that, writing $a_i:=A_i^{g_i}-\varphi$,
\begin{align*}
\|a_i\|_{L^2(U)}+\|\operatorname{tr}(a_i\wedge a_i\wedge a_i)\|_{L^1(U)}\rightarrow0.
\end{align*}
\end{enumerate}
Then all the conclusions of Theorem \ref{Theorem: uniqueness of tangent connections with isolated singularities} hold.

\end{corollary}

\begin{remark}
In particular, condition (ii) holds if the curvature norms converge almost everywhere on $U$ and
\begin{align*}
\sup_i\int_U|F_{A_i}|^{2+\varepsilon}<\infty
\end{align*}
for some $\varepsilon>0$. Condition (i) holds if $B_i\rightharpoonup\varphi$ weakly in $W^{1,2}(U)$. It also holds if $B_i\to\varphi$ strongly in $L^2(U)$ and $\sup_i\|B_i\|_{L^4(U)}<\infty$. For $\omega$-ASD connections, condition (iii) holds if $A_i^{g_i}\to\varphi$ strongly in $L^3(U)$.

These assumptions concern one annulus along one sequence of rescalings. Condition (i) measures the cost of interpolation and reduces to strong $L^2$ convergence for abelian differences. Condition (iii) is weaker than (i) by Cauchy--Schwarz. Under almost everywhere curvature convergence, condition (ii) expresses the absence of an annular energy defect. Remarks \ref{Remark: concrete HYM examples} and \ref{Remark: conformal examples} give families satisfying these conditions.
\end{remark}

Let $n\ge5$, let $\Omega\subset\R^n$ be open, and let $y\in\Omega$.
We say that a connection $A$ on a smooth principal $G$-bundle
$Q\to\Omega\smallsetminus\{y\}$, with finite energy near $y$ and
$W^{1,2}_{\mathrm{loc}}\cap L^4_{\mathrm{loc}}$ local connection forms, is
\textit{almost minimizing at $y$} if there exist $C_0,\alpha_0>0$ such that
\begin{align}\label{Equation: centered comparison all dimensions}
    \operatorname{YM}_{B_r(y)}(A)\le
    \operatorname{YM}_{B_r(y)}(B)+C_0r^{n-4+\alpha_0}
\end{align}
for almost every sufficiently small $r>0$ and every finite-energy
connection $B$ on $Q|_{B_r(y)\smallsetminus\{y\}}$ whose local
connection forms belong to $W^{1,2}\cap L^4$ on every annulus
$B_r(y)\smallsetminus\overline{B_s(y)}$, $0<s<r$, up to $\partial B_r(y)$, and whose
tangential trace on $\partial B_r(y)$ agrees with that of $A$.

\begin{theorem}[Uniqueness for almost minimizers]
\label{Theorem: uniqueness for almost minimizers}
    Let $G$ be a compact matrix Lie group, let $n\ge5$, let
    $\Omega\subset\R^n$ be open, and let $y\in\Omega$.
    Let $Q\to\Omega\smallsetminus\{y\}$ be a smooth principal $G$-bundle,
    and let $A$ be a connection on $Q$ almost minimizing at $y$ in the sense of
    \eqref{Equation: centered comparison all dimensions}.
    For rescalings at $y$, choose $R>0$ with $B_R(y)\Subset\Omega$ and
    a smooth radial identification
    $Q|_{B_R(y)\smallsetminus\{y\}}\simeq\pi_y^*P$,
    where $\pi_y(x):=(x-y)/|x-y|$ and $P$ is the pullback of $Q$
    by $\vartheta\mapsto y+R\vartheta$, $\vartheta\in\mathbb S^{n-1}$.
    Set $\pi(x):=x/|x|$ and $a(r):=\iota_{\mathbb S^{n-1}}^*A_{y,r}$. Assume that at least
    one of the following conditions holds:
    \begin{enumerate}[(i)]
    \item $n=5$, $\varphi=\pi^*a_0$, where $a_0$ is a smooth connection
    on $P$, and, for some sequence $\rho_i\downarrow0$ and some fixed
    annulus $U_0=B_t\smallsetminus\overline{B_s}$, $0<s<t$, there are
    $W^{1,2}$ bundle gauges $g_i$ such that
    \begin{align*}
        A_{y,\rho_i}^{g_i}\rightarrow\varphi,
        \qquad F_{A_{y,\rho_i}^{g_i}}\rightarrow F_\varphi
        \qquad\mbox{strongly in }L^2(U_0).
    \end{align*}
    \item There are $p>(n-1)/2$ and $M<\infty$ such that, for almost
    every sufficiently small $r>0$,
    \begin{align}\label{Equation: tangential curvature bound all dimensions}
        \|F_{a(r)}\|_{L^p(\mathbb S^{n-1})}\le M,
    \end{align}
    and $a(r)$ admits a $W^{1,p}$ representative through a $W^{1,2}$
    bundle gauge.
    \end{enumerate}
    Then there is a smooth Yang--Mills connection $a_0$ on $P$
    (the prescribed connection in case (i)) such that
    $\varphi:=\pi^*a_0$ is the unique tangent cone to $A$ at $y$,
    modulo gauge transformations. Moreover, there exists $\beta>0$ such that
    \begin{align*}
        d_2\big(a(r),a_0\big)\le C|\log r|^{-\beta}
    \end{align*}
    for almost every sufficiently small $r>0$, where $d_2$ is the
    $L^2$ distance modulo $W^{1,2}$ gauges. For every fixed compact
    annulus $U\Subset\mathbb R^n\smallsetminus\{0\}$, there are
    $W^{1,2}$ bundle gauges $h_r$ such that
    \begin{align*}
        \|A_{y,r}^{h_r}-\varphi\|_{L^2(U)}
        +\|F_{A_{y,r}^{h_r}}-F_\varphi\|_{L^2(U)}
        \le C_U|\log r|^{-\beta}.
    \end{align*}
    If the kernel of the second variation at $a_0$ is integrable, these
    estimates hold with $r^\beta$ in place of $|\log r|^{-\beta}$.
\end{theorem}

\begin{remark}
\label{Remark: trivial bundle comparison class}
\label{Remark: tangential curvature hypothesis}
    If $Q$ extends smoothly across $y$ and $A$ has a
    $W^{1,2}\cap L^4$ representative in a trivializing ball about $y$,
    the conclusion still holds when
    \eqref{Equation: centered comparison all dimensions} is required
    only for $W^{1,2}\cap L^4$ competitors on $Q|_{B_r(y)}$ with the same
    tangential boundary trace. Thus in this case one can use the
    comparison class of Definition
    \ref{Definition: almost YM-energy minimizers}, restricting the
    comparisons to balls centered at $y$.

    In case (ii), if $A$ is smooth away from $y$, the individual trace
    regularity is automatic, and the curvature assumption is precisely
    \begin{align}\label{Equation: physical tangential curvature bound}
        \operatorname*{ess\,sup}_{0<r<r_0}
        r^{2p-(n-1)}\int_{\partial B_r(y)}|F_A^{\mathrm{tan}}|^p<\infty,
        \qquad p>\frac{n-1}{2}.
    \end{align}
    Here $F_A^{\mathrm{tan}}$ denotes the restriction of $F_A$ to
    pairs of tangent vectors to $\partial B_r(y)$. The bound holds for
    smoothly asymptotically conical connections with smooth link; see
    Section \ref{sec: Luckhaus} for examples.
\end{remark}

\begin{remark}[Comparison with Yang's theorem]
\label{Remark: all dimensional Yang comparison}
\label{Remark: five dimensional Yang comparison}
Yang's theorem \cite{Yang} gives smooth tangent convergence for Yang--Mills connections with an isolated singularity and $|F_A(x)|\le C|x|^{-2}$. This bound implies the strong-tangent hypothesis of Theorem \ref{Theorem: uniqueness for almost minimizers}-(i) when $n=5$, as well as \eqref{Equation: physical tangential curvature bound}. Our almost-minimizing theorem therefore applies in Yang's setting whenever the corresponding centered comparison holds. The minimizing branch of Theorem \ref{Theorem: uniqueness of tangent connections with isolated singularities} also recovers the quadratic bound. In contrast, the almost minimizers in Example \ref{Example: radial shell perturbation} violate it. For nonclosed $\omega$, the $\omega$-ASD equation gives
\begin{align*}
    d_A(*F_A)=-F_A\wedge d\omega,
\end{align*}
so this class too includes equations beyond the Yang--Mills equation.
\end{remark}

\begin{remark} \label{remark: examples}
It would be interesting to determine when the logarithmic rate is optimal for nonintegrable Yang--Mills cones and to construct explicit examples, as in the minimal surface and harmonic map settings of \cite[Section 5]{AdamsSimon}, see also \cite{GulliverWhite}. The relation between analytic and algebraic tangent cones studied in \cite{ChenSunIMRN, ChenSunDuke, ChenSunGT, ChenSunInvent} provides a natural setting in which to investigate this question.
\end{remark}

In the special case in which the underlying manifold is K\"ahler $(X, \omega)$ and the vector bundle $(E, H)$ is Hermitian, more refined structural results have been obtained. For instance, the special class of \textit{admissible Hermitian Yang--Mills} connections has received a lot of interest in recent years, see \cite{ChenSunDuke} for the definition. For instance, in loc. cit. the authors are able to relate tangent cones of admissible Hermitian Yang--Mills connections at an isolated singularity to the complex algebraic geometry of the underlying reflexive sheaf (modulo an extra hypothesis on the Harder--Narasimhan--Seshadri filtration). See also \cite{JacobEnriqueWalpuski} for a proof in the case in which the vector bundle over $\mathbb{CP}^{n - 1}$ is a direct sum of polystable bundles. A crucial tool in \cite{ChenSunDuke} is the notion of \textit{algebraic tangent cone} already mentioned in Remark \ref{remark: examples}, and how it relates to the notion of (analytic) tangent cone that we introduced earlier in this introduction. Very loosely speaking this is a torsion-free sheaf over an exceptional divisor satisfying some extra requirements, and it carries information on the underlying singularities. In \cite{ChenSunGT} Chen and Sun were able to establish an algebro-geometric characterization of the bubbling set, always in the setting of isolated singularities (see also \cite{ChenSunIMRN}). They then go on and conclude with \cite{ChenSunInvent} by fully resolving this dichotomy between different notions of cones. We note that a key analytic tool of \cite{ChenSunDuke}, similar to our log-epiperimetric inequality, is a convexity result in the form of a \textit{three circle lemma}. We conclude this subsection by motivating the study of this class of connections. 
\begin{enumerate}[(i)]
\item From the complex geometric point of view, Bando and Siu \cite{BandoSiu} proved that polystable reflexive sheaves over a compact K\"ahler manifold always admit an admissible Hermitian Yang--Mills connection. This generalised the Donaldson--Uhlenbeck--Yau theorem for holomorphic vector bundles, meaning that these objects are relevant in algebraic geometry. 
\item From the gauge theoretic perspective, by \cite{Nakajima, tian-gauge-theory} admissible Hermitian Yang--Mills connections naturally arise in the compactification of the moduli space of smooth ones. Therefore, they play an important role in understanding the structure of the moduli space in gauge theory over higher dimensional K\"ahler manifolds. 
\item When doing gauge theory over $G_2$ manifolds, it is expected that singularities of this special class of connections in complex dimension three model singularities of $G_2$ instantons. We refer the reader to \cite{JacobWalpuski, EarpWalpuski} for further details on this. 
\end{enumerate}

The main ingredient in the uniqueness argument is a direct log-epiperimetric inequality for the Yang--Mills lagrangian, a quantitative estimate on the suboptimality of the homogeneous extension that we now describe, cf. Theorem \ref{Theorem: (log)-epiperimetric inequality for Yang--Mills functional} for the main estimate.  

\subsection{Epiperimetric Inequalities} Direct epiperimetric inequalities were introduced in seminal work of Reifenberg \cite{Reifenberg} in the context of minimal surfaces with the aim of proving that solutions of the Plateau problem, as posed by the author, were analytic \cite{ReifenbergAnalytic}. White later exploited this idea in \cite{White} to establish uniqueness of tangent cones for two-dimensional area-minimizing integral currents without boundary in $\mathbb{R}^n$. This result was then extended by Rivière in \cite{Riviere} where the author introduced the notion of \textit{lower} epiperimetric inequality, and proved the corresponding uniqueness of tangents. We note that as a consequence of this inequality, Rivière exhibited and investigated the phenomenon of \textit{splitting before tilting}, pivotal for the proof of regularity of $1-1$ integral currents joint with Tian \cite{RiviereTian}, and crucial for later developments on the regularity theory of area-minimizing and semicalibrated currents by De Lellis, Spadaro, and Spolaor, see \cite{CenterManifold, BranchedCenterManifold, Spolaor}. As part of a program to investigate the regularity of two-dimensional almost minimal currents, these last three authors recently established in \cite{DeLellisSpadaroSpolaor} an epiperimetric inequality in this setting, thus generalising the aforementioned work of White. 
 
The introduction of direct epiperimetric inequalities to the framework of free boundary problems is due to Spolaor and Velichkov in \cite{SpolaorVelichkov}. Subsequently, these last two authors, together with Colombo, proved a novel \textit{logarithmic epiperimetric inequality} in the context of obstacle-type problems \cite{ColomboSpolaorVelichkov1, ColomboSpolaorVelichkov2}. In a nutshell, this is a quantitative estimate on the optimality of the homogeneous extension which gives a logarithmic decay to the blow-up. The additional terms in the inequality are due to the possible presence of elements in the kernel of a suitable linearized operator. In \cite{ESVduke}, Spolaor and Velichkov, together with Engelstein, developed a new approach for proving this inequality based on reducing it to a quantitative estimate for a functional defined on the unit sphere, and studying the corresponding gradient flow. This was done for the Alt-Caffarelli functional, but the new perspective found fruitful applications to the study of multiplicity-one stationary cones with isolated singularities \cite{EngelsteinSpolaorVelichkov}, yielding new $\varepsilon$-regularity results for almost minimizers. See also \cite{SpolaorVelichkovEng} and \cite{SymmetricEdelenSpolaorVelichkov}. 

All of the results mentioned above are direct in the sense that they are based on an explicit construction of a competitor. This is usually more adapted to establishing decay estimates around singular points. However, a large class of epiperimetric inequalities are proven by contradiction. These are based on linearization techniques and the contradiction arguments appearing in the literature apply to regular points or singular points with additional structural hypotheses. In the setting of minimal submanifolds we mention works of Taylor on area-minimizing flat chains modulo 3 and $(\mathbf{M}, \varepsilon, \delta)$-minimizers \cite{TaylorFlatChains,  TaylorSoapBubble, TaylorEllipsoidal}, while for free boundary problems the first instance of an epiperimetric inequality is due to Taylor in \cite{TaylorCapillarity}. Later on, Weiss in \cite{Weiss} introduced an epiperimetric inequality in the setting of the classical obstacle problem at flat singular points and along the top stratum of the singular set. On the other hand, for the thin obstacle problem, we refer the reader to \cite{FocardiSpadaro, GarofaloPetrosyanMariana}. Eventually, for free boundary problems for harmonic measures we mention work of Badger, Engelstein, and Toro \cite{Badger} whose great novelty is to apply an epiperimetric inequality for functions that do not minimize any energy. Our main contribution to this line of investigation in the setting of the Yang--Mills Lagrangian is the following. 

\begin{theorem}[Log-epiperimetric inequality]\label{Theorem: (log)-epiperimetric inequality for Yang--Mills functional}
    Let $G$ be a compact matrix Lie group with Lie algebra $\mathfrak{g}$, let $n\ge 5$, and fix $\alpha\in(0,1)$. Let $P\to\mathbb S^{n-1}$ be a smooth principal $G$-bundle and let $A_0$ be a smooth Yang--Mills connection on $P$. Define $\tilde A_0$ to be its $0$-homogeneous extension to the radial pullback bundle over $\mathbb B^n\smallsetminus\{0\}$, given by
    \begin{align*}
        \tilde A_0:=\bigg(\frac{\cdot}{\lvert\,\cdot\,\rvert}\bigg)^*A_0.
    \end{align*}
    There exist constants $\varepsilon, \delta > 0$, and $\gamma \in [0, 1)$ depending on $n$, $G$, $\alpha$ and $A_0$ such that the following holds. If $A$ is a $C^{2,\alpha}$ connection on $P$ such that
    \begin{align*}
        \|A-A_0\|_{C^{2,\alpha}(\mathbb{S}^{n-1})}<\delta
    \end{align*}
    then there exists a finite-energy connection $\hat A$ on the radial pullback bundle, with $W^{1,2}\cap L^4$ local connection forms in smooth angular trivializations, such that $\iota_{\mathbb{S}^{n-1}}^*\hat A=A$ and 
    \begin{align} \label{equation: main epiperimetric inequality}
        \mathscr{Y}_{\mathbb{B}^n}(\hat A\,;\tilde A_0)\le\big(1-\varepsilon\lvert\mathscr{Y}_{\mathbb{B}^n}(\tilde A\,;\tilde A_0)\rvert^{\gamma}\big)\mathscr{Y}_{\mathbb{B}^n}(\tilde A\,;\tilde A_0),
    \end{align}
    where $\tilde A$ is the $0$-homogeneous extension of $A$. Norms of differences of connections are computed using $A_0$. Furthermore, if the kernel of the second variation is integrable, we can take $\gamma = 0$. 
\end{theorem}
\begin{remark}
    We expect that modifying the proof of Theorem \ref{Theorem: (log)-epiperimetric inequality for Yang--Mills functional} as in \cite{SymmetricEdelenSpolaorVelichkov} to prove a symmetric log-epiperimetric inequality, one can prove uniqueness of tangent cones at infinity for sequences of connections. In particular, one would appeal to the decay-growth Theorem à la Edelen-Spolaor-Velichkov, cf. \cite[Theorem 2.2]{SymmetricEdelenSpolaorVelichkov}.
\end{remark}

\subsection{Ideas of the proofs and structure of the article} The proof of Theorem \ref{Theorem: (log)-epiperimetric inequality for Yang--Mills functional} follows the strategy outlined in \cite{EngelsteinSpolaorVelichkov, ESVduke}. In particular, it relies on a careful construction of a competitor function with energy smaller than the one of the 0-homogeneous extension of $A$. The starting point is a slicing lemma to write the energy discrepancy $\mathscr{Y}_{\mathbb{B}^n}(\cdot\,;\cdot)$ in Theorem \ref{Theorem: (log)-epiperimetric inequality for Yang--Mills functional} in a more convenient form, cf. Lemma \ref{Lemma: slicing lemma}. This is done in Section \ref{sec: slicing and LS reduction}, where we also recall the Lyapunov--Schmidt reduction adapted to our setting, cf. Lemma \ref{Lemma: Lyapunov-Schmidt reduction}. An additional difficulty that we have to overcome here is the kernel of the second variation being infinite dimensional. This is due to the gauge invariance of the Yang--Mills lagrangian. A similar issue was faced in \cite{ColdingMinicozziInvent}, where the authors have to quotient out the diffeomorphism invariance of their functional $\mathcal{R}$ (they do so by the Ebin--Palais slice theorem). Similarly, Simon worked with normal graphs to avoid this issue \cite{AsymptoticSimon}, while Yang \cite{Yang} introduced a form of transverse gauge. We resolve this by simply working in a suitable relative Coulomb gauge. 

This way of writing the energy discrepancy suggests constructing the competitor by moving inwards along directions which decrease the energy around the critical connection $A_0$. On the relative Coulomb slice, the second variation has compact resolvent and a finite-dimensional kernel. After the Lyapunov--Schmidt reduction, we decrease the positive spectral components and increase the negative ones by fixed small amounts. Both changes decrease the quadratic energy. On the kernel, we follow the normalized gradient flow of the reduced analytic functional. The finite-dimensional \L ojasiewicz inequality quantifies the decrease and produces the exponent in \eqref{equation: main epiperimetric inequality}.

The proof appears in Section \ref{sec: proof of epiperimetric ineq}. In the integrable case, every Jacobi field in the Coulomb slice arises as the derivative of a family of nearby Yang--Mills connections, and the reduced functional vanishes identically near the origin. The proof then gives a classical epiperimetric inequality; see Subsection \ref{subsec: integrable case}. For comparison with the more restrictive notion of integrability through rotations for stationary varifolds, see \cite[Remark 1.3]{EngelsteinSpolaorVelichkov}.

\begin{remark}
An appropriate \L ojasiewicz--Simon inequality also provides a route to a log-epiperimetric inequality by flowing the full trace; see \cite[Proposition 3.1]{colombo-spolaor-velichkov-3}. Such inequalities for the Yang--Mills functional are available in several Sobolev topologies, including in higher dimensions; see \cite{Feehan, FeehanMaridakis, FeehanMaridakisCrelle}. Here the finite-dimensional reduction gives an explicit competitor and isolates the role of the kernel.
\end{remark}

The proof of Theorem \ref{Theorem: uniqueness of tangent connections with isolated singularities} appears in Section \ref{sec: proof of uniqueness}. Combining the log-epiperimetric inequality with almost monotonicity gives a differential inequality for the corrected energy excess. The radial curvature term yields a finite-length estimate, while annular compactness propagates closeness to the tangent cone. The resulting energy decay gives uniqueness by Proposition \ref{Proposition: Dini decrease implies uniqueness of tangent cone}. In dimension $5$, the embedding $W^{1,2}(\mathbb S^4)\hookrightarrow L^4(\mathbb S^4)$ permits the same construction in the critical topology. Together with Lemma \ref{Lemma: critical energy distance}, this proves Theorem \ref{Theorem: uniqueness for almost minimizers}-(i). In higher dimensions, the construction holds in $W^{1,2}\cap L^{n-1}$. Spherical gauge compactness and interpolation first produce a smooth strong tangent under the tangential curvature bound; Lemma \ref{Lemma: curvature controlled energy distance} then yields Theorem \ref{Theorem: uniqueness for almost minimizers}-(ii). Section \ref{sec: Luckhaus} treats non-concentration and examples.
\subsection*{Acknowledgements}
Both authors would like to thank Tristan Rivière, Luca Spolaor, and Song Sun for useful discussions, and Paul Feehan for helpful correspondence. We also would like to thank Jean E. Taylor for pointing out relevant references on the subject matter that we were not initially aware of. This material is based upon work supported by the National Science Foundation under Grant No. DMS-1928930, while the authors were in
residence at the Simons Laufer Mathematical Sciences Institute
(formerly MSRI) in Berkeley, California, during the Fall 2024
semester. D.P. acknowledges the support of the AMS-Simons travel grant. 

\section{Preliminaries on the Yang--Mills functional} \label{sec: preliminaries}
In this section we collect the basic definitions and properties of the Yang--Mills lagrangian. In particular, we prove that this functional is analytic. We also introduce the class of almost minimizers of the Yang--Mills energy, and define $\omega$-anti-self-dual connections. Under the algebraic normalization \eqref{Equation: algebraic bound for omega}, we then prove that the latter belong locally to the former. We prove an almost monotonicity formula resembling the one for semicalibrated currents. We conclude by explaining the phenomenon of concentration appearing in Theorem \ref{Theorem: uniqueness of tangent connections with isolated singularities}. 
\subsection{The Yang--Mills lagrangian and Yang--Mills connections}
We write the local expressions below in smooth trivializations of the bundle; local comparisons are made with connections on the restricted bundle having the same tangential boundary trace.
\begin{definition}[The Yang--Mills functional]
    Let $G$ be a compact matrix Lie group with Lie algebra $\mathfrak{g}$ and let $n\ge 2$ and $(N,h)$ be a smooth $n$-dimensional Riemannian manifold, possibly with smooth boundary $\partial N$.\\
    The \textit{Yang--Mills functional} $\operatorname{YM}_N:(W^{1,2}\cap L^4)(N,T^*N\otimes\mathfrak{g})\to[0,+\infty)$ on the trivial bundle over $N$ is given by 
    \vspace{-1mm}
    \begin{align*}
        \operatorname{YM}_N(A):=\int_{N}\lvert F_A\rvert^2\, d\vol_h \qquad\forall A\in(W^{1,2}\cap L^4)(N,T^*N\otimes\mathfrak{g}),
    \end{align*}
    \vspace{-1mm}
    where 
    \begin{align*}
        F_A:=dA+A\wedge A\in L^2(N,\wedge^2T^*N\otimes\mathfrak{g}).
    \end{align*}
    Given any open subset $U\subset N$, for every $A\in(W^{1,2}\cap L^4)(N,T^*N\otimes\mathfrak{g})$ we let
    \begin{align*}
        \operatorname{YM}_N(A\,;U):=\int_{U}\lvert F_A\rvert^2\, d\vol_h
    \end{align*}
    be the \textit{Yang--Mills energy of $A$ localized in $U$}.
\end{definition}
\begin{definition}[Yang--Mills connections]
    Let $G$ be a compact matrix Lie group with Lie algebra $\mathfrak{g}$. Let $n\ge 2$ and let $(N,h)$ be a smooth $n$-dimensional Riemannian manifold, possibly with smooth boundary $\partial N$. 
    \newline
    A \textit{Yang--Mills} connection on the trivial bundle over $N$ is a critical point of $\operatorname{YM}_N$ with respect to compactly supported variations in the interior of $N$.
\end{definition}
\begin{definition}[$\operatorname{YM}$-energy discrepancy]
    Let $G$ be a compact matrix Lie group with Lie algebra $\mathfrak{g}$. Let $n\ge 2$ and let $(N,h)$ be a smooth $n$-dimensional Riemannian manifold, possibly with smooth boundary $\partial N$. Let $A_0\in(W^{1,2}\cap L^4)(N,\wedge ^1T^*N\otimes\mathfrak{g})$. The functional $\mathscr{Y}_{N}(\,\cdot\,\,;A_0)$ given by
    \begin{align*}
        \mathscr{Y}_N(A\,;A_0):=\operatorname{YM}_N(A)-\operatorname{YM}_N(A_0)\qquad\forall\,A\in(W^{1,2}\cap L^4)(N,\wedge ^1T^*N\otimes\mathfrak{g})
    \end{align*}
    is called \textit{$\operatorname{YM}$-energy discrepancy with respect to $A_0$} on $N$.
\end{definition}
\begin{proposition} \label{proposition: analyticity}
    Let $G$ be a compact matrix Lie group with Lie algebra $\mathfrak{g}$. Let $n\ge 2$ and let $(N,h)$ be a smooth $n$-dimensional Riemannian manifold, possibly with smooth boundary $\partial N$. Given any $A_0\in(W^{1,2}\cap L^4)(N,\wedge ^1T^*N\otimes\mathfrak{g})$, the following facts hold.
    \begin{enumerate}[(i)]
        \item The functional $\operatorname{YM}_N$ is a polynomial functional of degree at most four on $(W^{1,2}\cap L^4)(N,T^*N\otimes\mathfrak{g})$, i.e. given any $A\in (W^{1,2}\cap L^4)(N,T^*N\otimes\mathfrak{g})$ for every $k=0, 1, 2, 3, 4$ there exists a continuous symmetric $k$-linear form
        \begin{align*}
            \nabla^k\operatorname{YM}_N(A):(W^{1,2}\cap L^4)(N,T^*N\otimes\mathfrak{g})^k\to\R
        \end{align*}
        such that 
        \begin{align*}
            \operatorname{YM}_N(A+\varphi)=\sum_{k=0}^4\frac{\nabla^k\operatorname{YM}_N(A)}{k!}[\underbrace{\varphi,...,\varphi}_{k\ \mbox{times}}] \qquad\forall\,\varphi\in (W^{1,2}\cap L^4)(N,T^*N\otimes\mathfrak{g}).
        \end{align*}
        \item The functional $\mathscr{Y}_N(\,\cdot\,\,;A_0)$ is real-analytic on $(W^{1,2}\cap L^4)(N,T^*N\otimes\mathfrak{g})$ and its first and second Fr\'echet differentials are given by\footnote{Here and throughout, by $d_A$ we denote the \textit{exterior covariant derivative} with respect to the connection $A$, given by
        \begin{align*}
            d_A\alpha:=d\alpha+[A\wedge\alpha]=d\alpha+A\wedge\alpha-(-1)^k\alpha\wedge A\qquad\text{for a $k$-form $\alpha$.}
        \end{align*}
        Moreover, we will denote by $d_A^*$ the formal $L^2$-adjoint operator of $d_A$.}
        \begin{align*}
            \nabla\mathscr{Y}_N(A\,;A_0)[\varphi]&=2\int_{N}\langle F_A,d_A\varphi\rangle\, d\vol_h\\
            \nabla^2\mathscr{Y}_N(A\,;A_0)[\varphi,\psi]&=2\int_{N}(\langle d_A\varphi,d_A\psi\rangle+\langle F_A,[\varphi\wedge\psi]\rangle)\, d\vol_h 
        \end{align*}
        for every $\varphi,\psi\in(W^{1,2}\cap L^4)(N,T^*N\otimes\mathfrak{g})$.
        \end{enumerate}
    \begin{proof}
        Since $\mathscr{Y}_N(\,\cdot\,\,;A_0)$ is simply a shift of $\operatorname{YM}_N$ by the constant additive factor $\operatorname{YM}_N(A_0)$, (ii) follows directly from (i). Hence, we turn to show (i). Fix any $A\in(W^{1,2}\cap L^4)(N,T^*N\otimes\mathfrak{g})$. Then, by direct computation, for every $\varphi\in(W^{1,2}\cap L^4)(N,T^*N\otimes\mathfrak{g})$ we have 
        \begin{align}\label{Equation: expansion YM near a point}
            \nonumber
            \operatorname{YM}_{N}(A+\varphi)&=\int_{N}\lvert F_{A+\varphi}\rvert^2\, d\vol_h\\
            \nonumber
            &=\int_{N}\lvert F_A+d_A\varphi+\varphi\wedge\varphi\rvert^2\, d\vol_h\\
            \nonumber
            &=\int_{N}\lvert F_A\rvert^2\,d\vol_h+2\int_{N}\langle F_A,d_A\varphi\rangle\, d\vol_h+\int_{N}\big(\lvert d_A\varphi\rvert^2+\langle F_A,[\varphi\wedge\varphi]\rangle\big)\, d\vol_h\\
            &\quad+2\int_{N}\langle d_A\varphi,\varphi\wedge\varphi\rangle\, d\vol_h+\int_{N}\lvert\varphi\wedge\varphi\rvert^2\, d\vol_h.
        \end{align}
        Let now
        \begin{align*}
            &\nabla\operatorname{YM}_{N}(A):(W^{1,2}\cap L^4)(N,T^*N\otimes\mathfrak{g})\to\R\\
            &\nabla^2\operatorname{YM}_{N}(A):(W^{1,2}\cap L^4)(N,T^*N\otimes\mathfrak{g})^2\to\R\\
            &\nabla^3\operatorname{YM}_{N}(A):(W^{1,2}\cap L^4)(N,T^*N\otimes\mathfrak{g})^3\to\R\\
            &\nabla^4\operatorname{YM}_{N}(A):(W^{1,2}\cap L^4)(N,T^*N\otimes\mathfrak{g})^4\to\R\\
        \end{align*}
        be given by
        \begin{align*}
            \nabla\operatorname{YM}_{N}(A)[\varphi]&:=2\int_{N}\langle F_A,d_A\varphi\rangle\, d\vol_h
        \end{align*}
        for every $\varphi\in (W^{1,2}\cap L^4)(N,T^*N\otimes\mathfrak{g})$,
        \begin{align*}
            \frac{\nabla^2\operatorname{YM}_{N}(A)}{2!}[\varphi_1,\varphi_2]&:=\int_{N}\big(\langle d_A\varphi_1,d_A\varphi_2\rangle+\langle F_A,[\varphi_1\wedge\varphi_2]\rangle\big)\, d\vol_h
        \end{align*}
        for every $\varphi_1,\varphi_2\in (W^{1,2}\cap L^4)(N,T^*N\otimes\mathfrak{g})$,
        \begin{align*}
            \frac{\nabla^3\operatorname{YM}_{N}(A)}{3!}[\varphi_1,\varphi_2,\varphi_3]
            &:=\frac13\int_N\big(\langle d_A\varphi_1,[\varphi_2\wedge\varphi_3]\rangle
            +\langle d_A\varphi_2,[\varphi_1\wedge\varphi_3]\rangle\\
            &\hspace{32mm}+\langle d_A\varphi_3,[\varphi_1\wedge\varphi_2]\rangle\big)\,d\vol_h
        \end{align*}
        for every $\varphi_1,\varphi_2,\varphi_3\in (W^{1,2}\cap L^4)(N,T^*N\otimes\mathfrak{g})$ and
        \begin{align*}
            \frac{\nabla^4\operatorname{YM}_{N}(A)}{4!}[\varphi_1,\varphi_2,\varphi_3,\varphi_4]
            &:=\frac1{12}\int_N\big(\langle[\varphi_1\wedge\varphi_2],[\varphi_3\wedge\varphi_4]\rangle\\
            &\hspace{12mm}+\langle[\varphi_1\wedge\varphi_3],[\varphi_2\wedge\varphi_4]\rangle
            +\langle[\varphi_1\wedge\varphi_4],[\varphi_2\wedge\varphi_3]\rangle\big)\,d\vol_h
        \end{align*}
        for every $\varphi_1,\varphi_2,\varphi_3,\varphi_4\in (W^{1,2}\cap L^4)(N,T^*N\otimes\mathfrak{g})$. Notice that $\nabla^k\operatorname{YM}_{N}(A)$ is a $k$-linear and continuous symmetric multilinear form on $(W^{1,2}\cap L^4)(N,T^*N\otimes\mathfrak{g})$ for every $k=1,...,4$. Moreover, by plugging the definitions of the operators $\nabla^k\operatorname{YM}_{N}(A)$ in \eqref{Equation: expansion YM near a point}, we get
        \begin{align*}
            \operatorname{YM}_{N}(A+\varphi)&=\operatorname{YM}_{N}(A)+\nabla\operatorname{YM}_{N}(A)[\varphi]+\frac{\nabla^2\operatorname{YM}_{N}(A)}{2}[\varphi,\varphi]\\
            &\quad+\frac{\nabla^3\operatorname{YM}_{N}(A)}{3!}[\varphi,\varphi,\varphi]+\frac{\nabla^4\operatorname{YM}_{N}(A)}{4!}[\varphi,\varphi,\varphi,\varphi]
        \end{align*}
        for every $\varphi\in(W^{1,2}\cap L^4)(N,T^*N\otimes\mathfrak{g})$. The statement follows.
    \end{proof}
\end{proposition}
\subsection{Almost $\operatorname{YM}$-energy minimizers and $\omega$-ASD connections}
\begin{definition}[Almost $\operatorname{YM}$-energy minimizers]\label{Definition: almost YM-energy minimizers}
    Let $G$ be a compact matrix Lie group with Lie algebra $\mathfrak{g}$. Let $n\ge 2$ and let $(N,h)$ be a smooth $n$-dimensional Riemannian manifold, possibly with smooth boundary $\partial N$. We say that $A\in (W^{1,2}\cap L^4)(N,T^*N\otimes\mathfrak{g})$ is an \textit{almost $\operatorname{YM}$-energy minimizer} if there exist $C,\alpha,\rho_0>0$ such that for every geodesic open ball $\mathscr{B}\subset N$ of radius $0<\rho<\rho_0$ such that $\overline{\mathscr{B}}\cap\partial N=\emptyset$ we have 
    \begin{align}\label{Equation: almost minimality property}
        \operatorname{YM}_N(A\,;\mathscr{B})\le\operatorname{YM}_N(\tilde A\,;\mathscr{B})+C\rho^{n-4+\alpha},
    \end{align}
    for every $\tilde A\in (W^{1,2}\cap L^4)(\mathscr{B},T^*\mathscr{B}\otimes\mathfrak{g})$ with $\iota_{\partial \mathscr{B}}^*\tilde A=\iota_{\partial\mathscr{B}}^*A$.\\
    If the previous inequality holds with $C=0$, then we say that $A$ is a \textit{$\operatorname{YM}$-energy minimizer}. 
\end{definition}
\begin{definition}[$\omega$-ASD connections]\label{Definition: omega-ASD connections}
    Let $G$ be a compact matrix Lie group with Lie algebra $\mathfrak{g}$. Let $n\ge 4$ and let $(N,h)$ be a smooth, oriented $n$-dimensional Riemannian manifold, possibly with smooth boundary $\partial N$. Let $\omega\in C^{\infty}(N,\wedge^{n-4}T^*N)$ be a smooth $(n-4)$-form on $N$. Its comass is
    \begin{align*}
        \|\omega\|_{*}:=\sup\{|\omega_x(e_1,...,e_{n-4})| \mbox{ : } x\in N,\,e_1,...,e_{n-4}\in T_xN \mbox{ with } \lvert e_1\wedge...\wedge e_{n-4}\rvert_h=1\}.
    \end{align*}
    We say that $A\in (W^{1,2}\cap L^4)(N,T^*N\otimes\mathfrak{g})$ is an \textit{$\omega$-anti-self-dual connection} (or, for short, an \textit{$\omega$-ASD connection}) if $A$ satisfies the following first order system of PDEs
    \begin{align}\label{Equation: asd equation}
        *F_A=-F_A\wedge\omega.
    \end{align}
\end{definition}
\begin{proposition}[$\omega$-ASD connections are almost minimizers]\label{Proposition: omega-ASDs are semicalibrated cycles}
    Let $n\ge 4$, let $(N,h)$ be a smooth, oriented Riemannian manifold, and let $\omega\in C^{0,\beta}_{\loc}(N,\wedge^{n-4}T^*N)$ for some $0<\beta\le 1$. Assume that
    \begin{align}\label{Equation: algebraic bound for omega}
        \langle\xi,*(\xi\wedge\omega)\rangle\ge-\lvert\xi\rvert^2
        \qquad\forall\,\xi\in\wedge^2T_x^*N,\quad x\in N.
    \end{align}
    Then every $\omega$-ASD connection $A\in(W^{1,2}\cap L^4)_{\loc}(N,T^*N\otimes\mathfrak g)$ is locally an almost $\operatorname{YM}$-energy minimizer, with exponent $\alpha=\beta$.
    \begin{proof}
        Work in normal coordinates centered at $y$ and write $\omega_y$ for the differential form whose coefficients are constant and equal to those of $\omega$ at $y$. For a sufficiently small geodesic ball $\mathscr B=B_r(y)$, let $\tilde A$ have the same tangential trace as $A$ on $\partial\mathscr B$. The relative Chern--Weil identity gives
        \begin{align}\label{Equation: relative Chern Weil identity}
            \int_{\mathscr B}\operatorname{tr}(F_A\wedge F_A)\wedge\omega_y
            =\int_{\mathscr B}\operatorname{tr}(F_{\tilde A}\wedge F_{\tilde A})\wedge\omega_y.
        \end{align}
        Indeed, for smooth connections, the difference of the two Chern--Weil forms is
        \begin{align*}
            d\left(2\int_0^1\operatorname{tr}\big((\tilde A-A)\wedge F_{A+t(\tilde A-A)}\big)\,dt\right),
        \end{align*}
        and its transgression has zero tangential trace. Since $d\omega_y=0$, Stokes' theorem proves \eqref{Equation: relative Chern Weil identity}. The identity extends to the stated Sobolev class by strong approximation in $W^{1,2}\cap L^4$, approximating $\tilde A-A$ by smooth one-forms with zero tangential trace.

        The $\omega$-ASD equation and \eqref{Equation: algebraic bound for omega} imply, respectively,
        \begin{align*}
            \operatorname{tr}(F_A\wedge F_A)\wedge\omega&=\lvert F_A\rvert^2\,d\vol_h,\\
            \operatorname{tr}(F_{\tilde A}\wedge F_{\tilde A})\wedge\omega&\le\lvert F_{\tilde A}\rvert^2\,d\vol_h.
        \end{align*}
        Here the invariant inner product is $\langle X,Y\rangle=-\operatorname{tr}(XY)$. Since $\lvert\omega-\omega_y\rvert\le Cr^\beta$, we obtain
        \begin{align}\label{Equation: relative almost minimality}
            (1-Cr^\beta)\operatorname{YM}_N(A\,;\mathscr B)
            \le(1+Cr^\beta)\operatorname{YM}_N(\tilde A\,;\mathscr B).
        \end{align}
        All constants can be chosen uniformly for centers in a fixed compact subset of $N$.

        Suppose first that $n\ge 5$. Set $f(r):=\operatorname{YM}_N(A\,;B_r(y))$. For almost every $r$, comparison with the radial extension of the boundary trace, together with \eqref{Equation: relative almost minimality}, gives
        \begin{align*}
            f(r)\le(1+Cr^\beta)\frac{r}{n-4}f'(r).
        \end{align*}
        In normal coordinates, the additional metric error is $O(r^2)$ and is absorbed into $Cr^\beta$. Hence, after increasing $C$,
        \begin{align*}
            \frac{d}{dr}\big(r^{4-n}f(r)\big)
            \ge-Cr^{\beta-1}r^{4-n}f(r)
        \end{align*}
        for almost every sufficiently small $r$. Integration shows that $f(r)\le Cr^{n-4}$, locally uniformly in $y$. If $n=4$, the same bound follows directly from the local finiteness of the energy.

        If $\operatorname{YM}_N(\tilde A\,;\mathscr B)\ge\operatorname{YM}_N(A\,;\mathscr B)$ there is nothing to prove. Otherwise, \eqref{Equation: relative almost minimality} and the preceding energy bound yield
        \begin{align*}
            \operatorname{YM}_N(A\,;\mathscr B)
            \le\operatorname{YM}_N(\tilde A\,;\mathscr B)+Cr^{n-4+\beta}.
        \end{align*}
        This is the desired almost minimality inequality.
    \end{proof}
\end{proposition}
\begin{remark}\label{Remark: omega-ASD connections are almost energy minimizers}
    Condition \eqref{Equation: algebraic bound for omega} is automatic in dimension $5$ for forms of comass at most one. It also holds for the standard forms defining Hermitian Yang--Mills, $G_2$-instanton, and $\operatorname{Spin}(7)$-instanton equations, with the convention in \eqref{Equation: asd equation}. For an almost Hermitian $2\ell$-manifold with fundamental form $\kappa$, take $\omega=\kappa^{\ell-2}/(\ell-2)!$; the eigenvalues of $\xi\mapsto*(\xi\wedge\omega)$ are $\ell-1$, $1$ and $-1$ on the corresponding nonzero summands. The standard $G_2$ and $\operatorname{Spin}(7)$ forms give eigenvalues $2,-1$ and $3,-1$, respectively. If $d\omega=0$, the same argument, using $\omega$ in place of $\omega_y$, gives exact local minimality.
\end{remark}
\begin{remark}\label{Remark: centered almost minimality}
    For conclusions at a fixed point $y$, it is enough that \eqref{Equation: algebraic bound for omega} hold at $y$. In Euclidean coordinates, the constant form $\omega_y$ then satisfies the algebraic bound everywhere. The same proof gives
    \begin{align*}
        (1-Cr^\beta)\operatorname{YM}(A;B_r(y))
        \le\operatorname{YM}(\tilde A;B_r(y))
    \end{align*}
    for competitors with the same tangential trace. The centered cone comparison again yields an $O(r^{n-4})$ energy bound and an additive error $Cr^{n-4+\beta}$. The comparison arguments below use this centered form; Lemma \ref{Lemma: torsion monotonicity} supplies monotonicity at all centers.
\end{remark}
\begin{lemma}[Almost monotonicity for $\omega$-ASD connections]\label{Lemma: torsion monotonicity}
Let $n\ge5$, let $\Omega\subset\mathbb R^n$ be open, let $\omega$ be a smooth $(n-4)$-form, and let $A\in(W^{1,2}\cap L^4)_{\loc}(\Omega,T^*\Omega\otimes\mathfrak g)$ satisfy \eqref{Equation: asd equation}. Assume that $A$ is smooth modulo gauge transformations outside a closed set of Lebesgue measure zero. For every compactly supported smooth vector field $X$, with flow $\Phi_t$, we have
\begin{align}\label{Equation: torsion stationarity}
\left.\frac{d}{dt}\right|_{t=0}\operatorname{YM}_\Omega(\Phi_t^*A)
=-\int_\Omega\operatorname{tr}(F_A\wedge F_A)\wedge\iota_Xd\omega.
\end{align}
Here the variation is supported on a compact subset, so only the local energy is differentiated. Consequently, on every relatively compact subdomain there is $c\ge0$ such that
\begin{align*}
e^{cr}r^{4-n}\int_{B_r(z)}|F_A|^2
\end{align*}
is non-decreasing for sufficiently small $r>0$, uniformly in the center $z$.
\begin{proof}
Set $C_A:=\operatorname{tr}(F_A\wedge F_A)$. Strong approximation of $A$ in $W^{1,2}\cap L^4$ and the Bianchi identity give $dC_A=0$ in the sense of distributions. The $\omega$-ASD equation gives
\begin{align*}
C_A\wedge\omega=|F_A|^2\,d\mathcal L^n,\qquad
2\operatorname{tr}(H\wedge F_A)\wedge\omega
=2\langle H,F_A\rangle\,d\mathcal L^n
\end{align*}
for every $\mathfrak g$-valued $2$-form $H$.
Differentiating the first identity in a smooth local gauge and using the second shows that
\begin{align*}
C_A\wedge\partial_j\omega=0
\end{align*}
on the regular set, and hence almost everywhere.

Write $M\cdot\eta$ for the derivation on covariant forms induced by an endomorphism $M$. Applying the second identity to $H=DX\cdot F_A$ gives
\begin{align*}
2\langle F_A,DX\cdot F_A\rangle\,d\mathcal L^n
=(DX\cdot C_A)\wedge\omega.
\end{align*}
The change-of-variables formula for the energy therefore yields
\begin{align*}
\left.\frac{d}{dt}\right|_{t=0}\operatorname{YM}_\Omega(\Phi_t^*A)
&=\int_\Omega\left(2\langle F_A,DX\cdot F_A\rangle-\operatorname{div}X\,|F_A|^2\right)\,d\mathcal L^n\\
&=-\int_\Omega C_A\wedge(DX\cdot\omega)
=-\int_\Omega C_A\wedge\mathcal L_X\omega.
\end{align*}
The last equality uses $C_A\wedge\partial_j\omega=0$. Cartan's formula and $dC_A=0$ prove \eqref{Equation: torsion stationarity}.
Moreover,
\begin{align*}
-\int_\Omega C_A\wedge\iota_Xd\omega
=2\int_\Omega\operatorname{tr}\big((\iota_XF_A)\wedge F_A\big)\wedge d\omega,
\end{align*}
whose absolute value is bounded by
\begin{align*}
C\|d\omega\|_{L^\infty(\operatorname{supp}X)}\int_\Omega|\iota_XF_A|\,|F_A|.
\end{align*}
Testing with $X(x)=\eta(|x-z|)(x-z)$, approximating the characteristic function of a ball by $\eta$, and discarding the nonnegative radial-curvature term gives
\begin{align*}
\frac{d}{dr}\left(r^{4-n}\int_{B_r(z)}|F_A|^2\right)
\ge-c\,r^{4-n}\int_{B_r(z)}|F_A|^2
\end{align*}
for almost every sufficiently small $r$. This is the usual monotonicity computation, cf. \cite{wentworth-chen}, and all its terms are locally integrable by \eqref{Equation: torsion stationarity}. Multiplication by $e^{cr}$ proves the conclusion.
\end{proof}
\end{lemma}

In the following we will need an almost monotonicity formula for almost $\operatorname{YM}$-energy minimizers on open subsets of $\R^n$ for $n\ge 5$. We will obtain such a formula by essentially following the argument developed in \cite[Proposition 2.1]{DeLellisSpadaroSpolaor} for almost minimizers of the area functional. Notice that an analogous monotonicity formula was obtained in \cite[Section 2]{wentworth-chen} for $\omega$-ASD connections. Furthermore, in the case of smooth Yang--Mills connections the same formula is essentially due to Price \cite{price-monotonicity-formula} and adapted by Tian in \cite[Theorem 2.1.2 and Remark 3]{tian-gauge-theory}.
\begin{proposition}[Almost monotonicity formula]\label{Proposition: almost monotonicity formula}
    Let $G$ be a compact matrix Lie group with Lie algebra $\mathfrak{g}$. Let $n\ge 5$ and let $\Omega\subset\mathbb R^n$ be an open set. Let $A\in(W^{1,2}\cap L^4)(\Omega,T^*\Omega\otimes\mathfrak g)$ be an almost $\operatorname{YM}$-energy minimizer. Then there exist $C,\alpha,\rho_0>0$ such that, for every $y\in\Omega$ and $0<\sigma<\rho<\min\{\rho_0,\operatorname{dist}(y,\partial\Omega)\}$,
    \begin{align*}
        \frac1{\rho^{n-4}}\int_{B_\rho(y)}|F_A|^2
        -\frac1{\sigma^{n-4}}\int_{B_\sigma(y)}|F_A|^2
        &+C(\rho^\alpha-\sigma^\alpha)\\
        &\ge\int_{B_\rho(y)\smallsetminus B_\sigma(y)}
        |x-y|^{4-n}|F_A\res\nu_y|^2\,d\mathcal L^n(x),
    \end{align*}
    where $\nu_y(x):=(x-y)/|x-y|$.
    \begin{proof}
        Fix $y\in\Omega$. For almost every sufficiently small $r>0$, the tangential trace of $A$ on $\partial B_r(y)$ belongs to $W^{1,2}\cap L^4$. Its radial extension
        \begin{align*}
            A_r:=\left(y+r\frac{\,\cdot\,-y}{|\,\cdot\,-y|}\right)^*\iota_{\partial B_r(y)}^*A
        \end{align*}
        belongs to $(W^{1,2}\cap L^4)(B_r(y),T^*B_r(y)\otimes\mathfrak g)$, since $n>4$, and has the same tangential boundary trace as $A$. Almost minimality and integration in polar coordinates give
        \begin{align}\label{Equation: almost monotonicity 1}
            \operatorname{YM}_\Omega(A\,;B_r(y))
            \le\frac r{n-4}\int_{\partial B_r(y)}|\iota_{\partial B_r(y)}^*F_A|^2\,d\mathscr H^{n-1}
            +C_0r^{n-4+\alpha}.
        \end{align}
        Set $f(r):=\int_{B_r(y)}|F_A|^2$. By the coarea formula, $f$ is absolutely continuous and
        \begin{align*}
            f'(r)=\int_{\partial B_r(y)}|F_A|^2\,d\mathscr H^{n-1}
        \end{align*}
        for almost every $r$. Using
        \begin{align*}
            |F_A|^2=|\iota_{\partial B_r(y)}^*F_A|^2+|F_A\res\nu_y|^2
        \end{align*}
        in \eqref{Equation: almost monotonicity 1}, we obtain
        \begin{align}\label{Equation: almost monotonicity 2}
            \frac{d}{dr}\left(r^{4-n}f(r)\right)
            \ge r^{4-n}\int_{\partial B_r(y)}|F_A\res\nu_y|^2\,d\mathscr H^{n-1}
            -(n-4)C_0r^{\alpha-1}
        \end{align}
        for almost every sufficiently small $r$. Integration over $(\sigma,\rho)$ and another application of the coarea formula give
        \begin{align}\label{Equation: almost monotonicity 3}
            \rho^{4-n}f(\rho)-\sigma^{4-n}f(\sigma)
            +\frac{(n-4)C_0}{\alpha}(\rho^\alpha-\sigma^\alpha)
            \ge\int_{B_\rho(y)\smallsetminus B_\sigma(y)}|x-y|^{4-n}|F_A\res\nu_y|^2\,d\mathcal L^n(x).
        \end{align}
        This proves the statement.
    \end{proof}
\end{proposition}
\begin{remark}
    We write
    \begin{align*}
        \Theta(r,y;A):=r^{4-n}\int_{B_r(y)}|F_A|^2,
        \qquad \Theta(y;A):=\lim_{r\downarrow0}\Theta(r,y;A).
    \end{align*}
    The limit exists by the proposition. Its proof uses only almost minimality on balls centered at $y$, so it also applies in the setting of Remark \ref{Remark: centered almost minimality}.
\end{remark}
\subsection{Concentration set}
We now explain the concentration phenomenon appearing in Theorem \ref{Theorem: uniqueness of tangent connections with isolated singularities}. Let $\{A_i\}$ be a sequence of smooth Yang--Mills connections on $\Omega$ with $\operatorname{YM}_\Omega(A_i)\le\Lambda$. By \cite[Proposition 3.1.2]{tian-gauge-theory}, after passing to a subsequence there is an admissible limiting connection $A$, with smooth convergence modulo local gauge transformations away from a closed concentration set. We fix such a subsequence and write
\begin{align*}
|F_{A_i}|^2\,d\mathcal L^n\rightharpoonup^*\mu
=|F_A|^2\,d\mathcal L^n+\nu,
\end{align*}
where $\nu$ is a nonnegative Radon measure. The concentration set may be described by
\begin{align*}
\Sigma:=\left\{x\in\Omega:\liminf_{r\downarrow0}r^{4-n}\mu(B_r(x))\ge\varepsilon_0\right\},
\end{align*}
where $\varepsilon_0$ is an $\varepsilon$-regularity threshold. It has locally finite $\mathscr H^{n-4}$ measure. The measure $\nu$ records the defect of strong curvature convergence, and its vanishing is equivalent to strong $L^2_{\loc}$ convergence of the curvatures in the corresponding gauge identifications. If the limit is stationary, its singular set has vanishing $\mathscr H^{n-4}$ measure after removal of removable singularities. See \cite{tian-gauge-theory,TaoTian} for the gauge-theoretic compactness statements and \cite{lin-gradient-estimates} for the analogous phenomenon for harmonic maps.

For admissible Hermitian Yang--Mills connections on a punctured K\"ahler ball, the concentration and bubbling sets have additional complex-analytic structure. We refer to \cite{ChenSunIMRN,ChenSunDuke,ChenSunGT,ChenSunInvent} for the precise statements and their relation to algebraic tangent cones.
\section{The slicing lemma and the Lyapunov--Schmidt reduction} \label{sec: slicing and LS reduction}
The aim of this section is to set the stage for the proof of Theorem \ref{Theorem: (log)-epiperimetric inequality for Yang--Mills functional} by proving the crucial slicing lemma, and recalling the classical Lyapunov-Schmidt reduction and adapting it to our setting. 
We let $P\to\mathbb S^{n-1}$ be a smooth principal $G$-bundle and $\pi(x):=x/|x|$. On $\pi^*P$, local connection forms are taken in smooth angular trivializations.
\begin{definition}
    Let $G$ be a compact matrix Lie group with Lie algebra $\mathfrak{g}$ and let $n\ge 5$. We say that a finite-energy connection $A$ on $\pi^*P\to\mathbb B^n\smallsetminus\{0\}$ with $W^{1,2}\cap L^4$ local connection forms is \textit{conical} if $\iota_{\mathbb S^{n-1}}^*A$ is a $W^{1,2}\cap L^4$ connection on $P$ and
    \begin{align*}
        A=\bigg(\frac{\cdot}{\lvert\,\cdot\,\rvert}\bigg)^*\iota_{\mathbb{S}^{n - 1}}^*A.
    \end{align*}
\end{definition}
\allowdisplaybreaks
\begin{lemma}[Slicing lemma]\label{Lemma: slicing lemma}
Let $G$ be a compact matrix Lie group with Lie algebra $\mathfrak g$ and let $n\ge5$. Let $A_0$ be a $W^{1,2}\cap L^4$ connection on $P$, and let $\tilde A_0:=\pi^*A_0$. Then every finite-energy connection $A$ on $\pi^*P$ with $W^{1,2}\cap L^4$ local connection forms satisfies
\begin{align*}
\mathscr Y_{\mathbb B^n}(A\,;\tilde A_0)
&=\int_0^1\mathscr Y_{\mathbb S^{n-1}}(\Psi_\rho^*A\,;A_0)\rho^{n-5}\,d\rho\\
&\quad+\int_0^1\int_{\mathbb S^{n-1}}
\big|\Psi_\rho^*(F_A\res\nu_0)\big|^2\,d\mathscr H^{n-1}\,\rho^{n-3}\,d\rho,
\end{align*}
where $\Psi_\rho:\mathbb S^{n-1}\to\partial B_\rho(0)$ is given by $\Psi_\rho(x)=\rho x$, and $\nu_0(x)=x/|x|$. In particular, if $A$ is conical, then
\begin{align*}
\mathscr Y_{\mathbb B^n}(A\,;\tilde A_0)
=\frac1{n-4}\mathscr Y_{\mathbb S^{n-1}}(\iota_{\mathbb S^{n-1}}^*A\,;A_0).
\end{align*}
\begin{proof}
For almost every $\rho\in(0,1)$, the restriction of $A$ to $\partial B_\rho(0)$ belongs to $W^{1,2}\cap L^4$ and its curvature is the tangential restriction of $F_A$. The orthogonal decomposition of $2$-forms into tangential and radial parts gives
\begin{align*}
|F_A|^2=|\iota_{\partial B_\rho(0)}^*F_A|^2+|F_A\res\nu_0|^2
\qquad\text{almost everywhere on }\partial B_\rho(0).
\end{align*}
The coarea formula and the scaling of $2$-forms and $1$-forms therefore yield
\begin{align*}
\operatorname{YM}_{\mathbb B^n}(A)
&=\int_0^1\rho^{n-5}\int_{\mathbb S^{n-1}}|F_{\Psi_\rho^*A}|^2\,d\mathscr H^{n-1}\,d\rho\\
&\quad+\int_0^1\rho^{n-3}\int_{\mathbb S^{n-1}}|\Psi_\rho^*(F_A\res\nu_0)|^2\,d\mathscr H^{n-1}\,d\rho.
\end{align*}
For $\tilde A_0$, the radial term vanishes and $\Psi_\rho^*\tilde A_0=A_0$. Subtracting its energy proves the identity. If $A$ is conical, its radial term also vanishes and its spherical trace is independent of $\rho$, which gives the last formula.
\end{proof}
\end{lemma}
\begin{remark}\label{Remark: ellipticity of the second variation of the YM-energy discrepancy}
Let $A_0$ be a smooth Yang--Mills connection on $P$. All Sobolev and H\"older norms below are computed using $A_0$. The gauge invariance of the Yang--Mills functional gives an infinite-dimensional kernel for its second variation on the full space of connections. To remove this degeneracy, we work on the affine Coulomb slice $A_0+X$, where
\begin{align*}
X:=\{a\in C^{2,\alpha}(\mathbb S^{n-1},T^*\mathbb S^{n-1}\otimes\mathfrak g_P):d_{A_0}^*a=0\}.
\end{align*}
Let $\mathcal H$ be the $L^2$-closure of $X$. At the critical point $A_0$, the Hessian preserves $\mathcal H$: indeed, gauge invariance implies that it annihilates $d_{A_0}\xi$ for every smooth section $\xi$ of $\mathfrak g_P$, and the Hessian is symmetric. Adding the gauge-fixing term $2d_{A_0}d_{A_0}^*$ gives a self-adjoint elliptic operator on all $\mathfrak g_P$-valued $1$-forms. Its restriction to $\mathcal H$ agrees with the Hessian of the functional on the Coulomb slice. Consequently, this restriction has compact resolvent, a finite-dimensional kernel and negative eigenspace, and a complete $L^2$-orthonormal basis of smooth eigenforms in $\mathcal H$.
\end{remark}
\begin{lemma}[Lyapunov--Schmidt reduction for the Yang--Mills functional]\label{Lemma: Lyapunov-Schmidt reduction}
Let $G$ be a compact matrix Lie group with Lie algebra $\mathfrak g$, let $n\ge5$, and let $A_0$ be a smooth Yang--Mills connection on $P$. Let $X$ and $\mathcal H$ be as in Remark \ref{Remark: ellipticity of the second variation of the YM-energy discrepancy}, and set
\begin{align*}
\mathcal E(a):=\mathscr Y_{\mathbb S^{n-1}}(A_0+a\,;A_0),\qquad
K:=\ker D^2\mathcal E(0)|_X.
\end{align*}
Let $K^\perp$ denote the orthogonal complement of $K$ in $\mathcal H$, and let $P_K$ and $P_{K^\perp}$ be the corresponding orthogonal projections. There exist a neighborhood $U$ of $0$ in $K$ and an analytic map $\Upsilon:U\to X\cap K^\perp$ such that
\begin{enumerate}[(i)]
\item $\Upsilon(0)=0$ and $D\Upsilon(0)=0$;
\item $P_{K^\perp}\nabla_X\mathcal E(k+\Upsilon(k))=0$ for every $k\in U$;
\item the analytic function $\mathfrak q(k):=\mathcal E(k+\Upsilon(k))$ satisfies
\begin{align*}
\nabla\mathfrak q(k)=P_K\nabla_X\mathcal E(k+\Upsilon(k));
\end{align*}
\item after shrinking $U$, there is $C>0$ such that
\begin{align*}
\|D\Upsilon(k)[\eta]\|_{C^{2,\alpha}(\mathbb S^{n-1})}
\le C\|\eta\|_{L^2(\mathbb S^{n-1})}
\qquad\text{for every }k\in U,\ \eta\in K.
\end{align*}
\end{enumerate}
Here $\nabla_X$ is the $L^2$-gradient projected onto $\mathcal H$.
\begin{proof}
The projected gradient is analytic from $X$ to the corresponding Coulomb subspace of $C^{0,\alpha}$. By Remark \ref{Remark: ellipticity of the second variation of the YM-energy discrepancy}, its derivative at $0$ is invertible on $K^\perp$. The analytic implicit function theorem applied to
\begin{align*}
(k,w)\mapsto P_{K^\perp}\nabla_X\mathcal E(k+w)
\end{align*}
gives (i) and (ii). The chain rule and (ii) give (iii), since $D\Upsilon(k)[\eta]\in K^\perp$. Finally, (iv) follows by continuity of $D\Upsilon$ and equivalence of norms on the finite-dimensional space $K$.
\end{proof}
\end{lemma}
\section{Proof of the log-epiperimetric inequality for Yang--Mills connections}\label{sec: proof of epiperimetric ineq}
We retain the notation of Lemma \ref{Lemma: Lyapunov-Schmidt reduction}. Let $A$ satisfy the assumptions of Theorem \ref{Theorem: (log)-epiperimetric inequality for Yang--Mills functional}. By the local Coulomb slice theorem, cf. \cite[Theorem 8.1]{wehrheim}, if $\delta>0$ is sufficiently small there is $h\in C^{3,\alpha}(\mathbb S^{n-1},G_P)$ such that
\begin{align*}
a:=A^h-A_0\in X,\qquad
\|a\|_{C^{2,\alpha}(\mathbb S^{n-1})}\le C\|A-A_0\|_{C^{2,\alpha}(\mathbb S^{n-1})}.
\end{align*}
We decompose
\begin{align*}
a=b+\Upsilon(b)+w,\qquad b:=P_Ka,\qquad
w:=P_{K^\perp}a-\Upsilon(b).
\end{align*}
Write $w=w_++w_-$ according to the positive and negative eigenspaces of the Hessian, and set
\begin{align*}
Q(w):=\frac12D^2\mathcal E(0)[w,w],\qquad
T:=Q(w_+)-Q(w_-).
\end{align*}
The spectral gap and elliptic estimates give constants $c,C>0$, depending only on $A_0$ and $n$, such that
\begin{align}\label{Equation: spectral norm equivalence}
c\|w\|_{W^{1,2}}^2\le T\le C\|w\|_{W^{1,2}}^2.
\end{align}
The projections onto $K$ and the negative eigenspace have finite rank and smooth image, so all components are small in $C^{2,\alpha}$.

The second variation gives
\begin{align}\label{Equation: Hessian continuity}
\big|D^2\mathcal E(z)[\zeta,\zeta]-D^2\mathcal E(0)[\zeta,\zeta]\big|
\le C\|z\|_{C^1}\|\zeta\|_{W^{1,2}}^2
\end{align}
whenever $\|z\|_{C^1}\le1$. Let $f:=\mathfrak q$. Taylor's formula and Lemma \ref{Lemma: Lyapunov-Schmidt reduction}-(ii) therefore imply
\begin{align}\label{Equation: reduced Taylor expansion}
\mathcal E(k+\Upsilon(k)+z)=f(k)+Q(z)+R(k,z),\qquad
|R(k,z)|\le C\delta\|z\|_{W^{1,2}}^2,
\end{align}
provided $z\in X\cap K^\perp$ and $k,\Upsilon(k),z$ have $C^{2,\alpha}$-norm at most $C\delta$.

We also recall the finite-dimensional \L ojasiewicz inequality, cf. \cite{Loj}.
\begin{lemma}\label{lemma: finite dimensional Lojasiewicz inequality}
Let $U\subset\mathbb R^\ell$ be open and let $f:U\to\mathbb R$ be analytic. At every critical point $x\in U$ there are a neighborhood $V$ of $x$, an exponent $\theta\in(0,1/2]$ and $c>0$ such that
\begin{align*}
|\nabla f(y)|\ge c|f(y)-f(x)|^{1-\theta}\qquad\text{for every }y\in V.
\end{align*}
\end{lemma}
We identify $K$ with a Euclidean space using an $L^2$-orthonormal basis. Since $f(0)=0$ and $Df(0)=0$, this lemma gives
\begin{align}\label{Equation: reduced Lojasiewicz inequality}
|\nabla f(k)|\ge c|f(k)|^{1-\theta}
\end{align}
near $0$. Put $p:=2-2\theta\in[1,2)$ and $f_+(b):=\max\{f(b),0\}$.

If $f(b)>0$, let $v$ follow the normalized negative gradient flow
\begin{align*}
v'(t)=-\frac{\nabla f(v(t))}{|\nabla f(v(t))|},\qquad v(0)=b,
\end{align*}
stopped at $f(v(t))=f(b)/2$ and continued constantly thereafter. For a small fixed $\kappa_f>0$, set
\begin{align*}
k(\rho):=v\big(\kappa_f f(b)^{1-\theta}(1-\rho)\big),\qquad 0\le\rho\le1.
\end{align*}
If $f(b)\le0$, set $k(\rho):=b$. Since $|v'|\le1$ and $|f(b)|\le C|b|^2$, the path stays in the reduction neighborhood for small $\delta$ and $\kappa_f$, and
\begin{align}\label{Equation: kernel flow estimates}
|k'(\rho)|&\le\kappa_f f_+(b)^{1-\theta},\nonumber\\
f(k(\rho))-f(b)&\le-c\kappa_f(1-\rho)f_+(b)^p.
\end{align}
Indeed, before stopping, \eqref{Equation: reduced Lojasiewicz inequality} bounds the rate of energy decrease from below by $c(f(b)/2)^{1-\theta}$. After stopping, the decrease is $f(b)/2$, which gives the same bound because $p\ge1$, $f(b)\le1$ and $\kappa_f\le1$.

Let $\kappa\in(0,1)$ be another small fixed constant and define
\begin{align*}
w(\rho)&:=(1-\kappa(1-\rho))w_++(1+\kappa(1-\rho))w_-,\\
a(\rho)&:=k(\rho)+\Upsilon(k(\rho))+w(\rho).
\end{align*}
Then
\begin{align}\label{Equation: quadratic decrease}
Q(w(\rho))-Q(w)
&=(-2\kappa(1-\rho)+\kappa^2(1-\rho)^2)Q(w_+)\nonumber\\
&\quad +(2\kappa(1-\rho)+\kappa^2(1-\rho)^2)Q(w_-)\nonumber\\
&\le-\kappa(1-\rho)T.
\end{align}
Define $B$ on $\pi^*P\to\mathbb B^n\smallsetminus\{0\}$ by requiring that it have no radial component and
\begin{align*}
\Psi_\rho^*B=A_0+a(\rho),
\end{align*}
and set
\begin{align*}
\hat A:=B^{\tilde h^{-1}},\qquad \tilde h:=\pi^*h.
\end{align*}
The angular $C^2$ bounds and bounded radial derivatives give $W^{1,2}\cap L^4$ local connection forms for $B,\hat A$ and $W^{2,2}\cap W^{1,4}$ local representatives of $\tilde h$, since $n\ge5$. The trace of $\hat A$ is $A$.

The slicing lemma and gauge invariance give
\begin{align}\label{equation: starting estimate}
\mathscr Y_{\mathbb B^n}(\hat A\,;\tilde A_0)
-\mathscr Y_{\mathbb B^n}(\tilde A\,;\tilde A_0)
&=\int_0^1\big(\mathcal E(a(\rho))-\mathcal E(a)\big)\rho^{n-5}\,d\rho\nonumber\\
&\quad+\int_0^1\|a'(\rho)\|_{L^2}^2\rho^{n-3}\,d\rho.
\end{align}
By Lemma \ref{Lemma: Lyapunov-Schmidt reduction}-(iv), \eqref{Equation: spectral norm equivalence} and \eqref{Equation: kernel flow estimates},
\begin{align*}
\int_0^1\|a'(\rho)\|_{L^2}^2\rho^{n-3}\,d\rho
\le C\kappa_f^2f_+(b)^p+C\kappa^2T.
\end{align*}
On the other hand, \eqref{Equation: reduced Taylor expansion}, \eqref{Equation: kernel flow estimates} and \eqref{Equation: quadratic decrease} yield
\begin{align*}
\int_0^1\big(\mathcal E(a(\rho))-\mathcal E(a)\big)\rho^{n-5}\,d\rho
\le-c\kappa_f f_+(b)^p-c\kappa T+C\delta T.
\end{align*}
Choosing $\kappa_f,\kappa$ small and then $\delta$ small, \eqref{equation: starting estimate} gives
\begin{align}\label{Equation: total competitor decrease}
\mathscr Y_{\mathbb B^n}(\hat A\,;\tilde A_0)
\le\mathscr Y_{\mathbb B^n}(\tilde A\,;\tilde A_0)-c\big(T+f_+(b)^p\big).
\end{align}
To conclude, put $E:=\mathscr Y_{\mathbb B^n}(\tilde A\,;\tilde A_0)$. If $E\le0$, the homogeneous extension itself satisfies the asserted inequality. If $E>0$, the slicing lemma and \eqref{Equation: reduced Taylor expansion} imply
\begin{align*}
E=\frac{\mathcal E(a)}{n-4}\le C\big(f_+(b)+T\big).
\end{align*}
Since $T\le1$ after decreasing $\delta$, we obtain
\begin{align*}
E^p\le C\big(f_+(b)^p+T\big).
\end{align*}
Combining this with \eqref{Equation: total competitor decrease} proves
\begin{align*}
\mathscr Y_{\mathbb B^n}(\hat A\,;\tilde A_0)
\le (1-\varepsilon |E|^\gamma)E,
\qquad \gamma:=p-1=1-2\theta\in[0,1).
\end{align*}
\subsection{The integrable case}\label{subsec: integrable case}
We recall the integrability condition of \cite{AdamsSimon}. The kernel $K$ is integrable if for every $v\in K$ there is a family $a_s\in X$ of critical points of $\mathcal E$ such that $a_s\to0$ in $C^{2,\alpha}$ and $a_s/s\to v$ in $L^2$ as $s\downarrow0$. Analyticity gives the following characterization; see \cite[Lemma 1]{AdamsSimon} or \cite[Lemma 2.3]{EngelsteinSpolaorVelichkov}.
\begin{lemma}
The kernel $K$ is integrable if and only if $f\equiv0$ in a neighborhood of $0$.
\end{lemma}
In this case we take $k(\rho)\equiv b$. The preceding construction gives a decrease of at least $cT$, while $E\le CT$ whenever $E>0$. Thus the conclusion holds with $\gamma=0$. The integrability condition means that the nearby critical connections in the Coulomb slice form a smooth finite-dimensional family whose tangent space at $A_0$ is $K$; passing to the moduli space also requires taking the quotient by the stabilizer of $A_0$.
\subsection{The critical Sobolev topology in dimension five}\label{subsec: critical Sobolev topology}
In dimension five the construction applies in $W^{1,2}$.
\begin{theorem}\label{Theorem: weak epiperimetric five dimensions}
Let $P\to\mathbb S^4$ be a smooth principal bundle with compact matrix structure group $G$, and let $A_0$ be a smooth Yang--Mills connection on $P$. Set $\pi(x):=x/|x|$ and $\tilde A_0:=\pi^*A_0$. There exist $\delta,\varepsilon>0$ and $\gamma\in[0,1)$ such that, whenever $A$ is a $W^{1,2}$ connection on $P$ satisfying
\begin{align*}
\|A-A_0\|_{W^{1,2}(\mathbb S^4)}<\delta,
\end{align*}
there is a connection $\hat A$ on $\pi^*P\to\mathbb B^5\smallsetminus\{0\}$ with trace $A$ and $W^{1,2}\cap L^4$ local connection forms such that
\begin{align*}
\mathscr Y_{\mathbb B^5}(\hat A\,;\tilde A_0)
\le (1-\varepsilon |E|^\gamma)E,
\qquad E:=\mathscr Y_{\mathbb B^5}(\pi^*A\,;\tilde A_0).
\end{align*}
The smallness assumption may also be imposed after a $W^{2,2}$ gauge transformation. If the kernel at $A_0$ is integrable, the conclusion holds with $\gamma=0$.
\begin{proof}
The critical relative Coulomb gauge theorem, cf. \cite[Theorem 14]{FeehanMaridakis} with $p=2$ and fixed $q>2$, gives
\begin{align}\label{Equation: critical Coulomb slice}
d_{A_0}^*(A^h-A_0)=0,\qquad
\|A^h-A_0\|_{W^{1,2}}\le C\|A-A_0\|_{W^{1,2}}.
\end{align}
For smooth $A$, the constants are independent of higher norms. For general $A$, take smooth $A_j\to A$ strongly in $W^{1,2}$. Their Coulomb representatives $B_j=A_j^{h_j}$ are bounded in $W^{1,2}\subset L^4$, and
\begin{align*}
dh_j=h_jB_j-A_jh_j
\end{align*}
bounds $h_j$ in $W^{1,4}$ and, after differentiation, in $W^{2,2}$. A weakly convergent subsequence also converges almost everywhere to a gauge $h$. The equation identifies the weak limit of $B_j$ with $A^h$, proving \eqref{Equation: critical Coulomb slice}.

On the slice
\begin{align*}
X:=\{a\in W^{1,2}(\mathbb S^4,T^*\mathbb S^4\otimes\g_P):d_{A_0}^*a=0\},
\end{align*}
the embedding $W^{1,2}\hookrightarrow L^4$ makes $\mathcal E$ analytic. Its Hessian has the same smooth kernel and spectral decomposition, and is an isomorphism from $X\cap K^\perp$ to its dual. The implicit function theorem therefore gives Lemma \ref{Lemma: Lyapunov-Schmidt reduction} in this topology; uniqueness identifies its reduction with the smooth one.

The same embedding gives
\begin{align}\label{Equation: critical Hessian continuity}
\big|D^2\mathcal E(z)[\zeta,\zeta]-D^2\mathcal E(0)[\zeta,\zeta]\big|
\le C\|z\|_{W^{1,2}}\|\zeta\|_{W^{1,2}}^2
\end{align}
for $\|z\|_{W^{1,2}}\le1$, so \eqref{Equation: reduced Taylor expansion} holds for small $W^{1,2}$ components. The spectral and flow estimates are unchanged. The resulting path is bounded in $W^{1,2}$ with radial derivative bounded in $L^2$, giving $W^{1,2}\cap L^4$ forms on the cone. Lemma \ref{Lemma: slicing lemma}, applied in angular trivializations, then gives \eqref{Equation: total competitor decrease} and the conclusion, including the integrable case. Finally, $h\in W^{2,2}(\mathbb S^4)\subset W^{1,4}$ extends radially in $W^{2,2}\cap W^{1,4}(\mathbb B^5)$, so gauging back preserves the competitor class.
\end{proof}
\end{theorem}

For connections on a fixed smooth bundle $P\to\mathbb S^{n-1}$, put
\begin{align}\label{Equation: weak gauge distance}
d_2(A,B):=\inf_{g\in W^{1,2}(\mathbb S^{n-1},G_P)}\|A^g-B\|_{L^2(\mathbb S^{n-1})}.
\end{align}
The gauge action is interpreted as an $L^2$ connection form.
\begin{lemma}\label{Lemma: critical energy distance}
Let $P\to\mathbb S^4$ be a smooth principal bundle with compact matrix structure group, and let $A_0$ be a smooth connection on $P$. For every $\delta>0$ there is $\eta>0$ such that any $W^{1,2}$ connection $A$ satisfying
\begin{align*}
d_2(A,A_0)<\eta,\qquad
\operatorname{YM}_{\mathbb S^4}(A)\le\operatorname{YM}_{\mathbb S^4}(A_0)+\eta
\end{align*}
admits a $W^{2,2}$ gauge $h$ such that
\begin{align*}
d_{A_0}^*(A^h-A_0)=0,\qquad
\|A^h-A_0\|_{W^{1,2}(\mathbb S^4)}<\delta.
\end{align*}
If $A$ is smooth, $h$ may be chosen smooth.
\begin{proof}
We first prove the assertion for smooth connections by its sequential formulation. Let $A_j$ be smooth, and choose $W^{1,2}$ gauges $g_j$ such that $a_j:=A_j^{g_j}$ converges to $A_0$ in $L^2$. Suppose also that
\begin{align*}
\limsup_j\operatorname{YM}_{\mathbb S^4}(A_j)\le\operatorname{YM}_{\mathbb S^4}(A_0).
\end{align*}
Strong Sobolev approximation gives the distributional identity $d(u\,dv)=du\wedge dv$ for $W^{1,2}$ matrix-valued maps, hence the Maurer--Cartan identity for $g_j$. The mixed products lie in $L^{4/3}$ and cancel, proving $F_{a_j}=g_j^{-1}F_{A_j}g_j\in L^2$. Set $b_j:=a_j-A_0$. Then
\begin{align*}
F_{a_j}-F_{A_0}=d_{A_0}b_j+b_j\wedge b_j\rightarrow0
\qquad\text{in distributions},
\end{align*}
since $b_j\to0$ in $L^2$. The energy bound gives weak $L^2$ convergence, and lower semicontinuity identifies the limiting norms. Therefore
\begin{align}\label{Equation: critical curvature strong convergence}
F_{a_j}\rightarrow F_{A_0}\qquad\text{strongly in }L^2.
\end{align}

By \eqref{Equation: critical curvature strong convergence}, choose a finite covering whose slightly larger balls have uniformly small curvature energy. Applying the local Coulomb gauge theorem with normal boundary condition to $A_j$ gives forms $B_j^\alpha$ with small $W^{1,2}$ norms, converging weakly in $W^{1,2}$ and strongly in $L^2$ to $B^\alpha$. The gauges from $a_j$ to $B_j^\alpha$ satisfy
\begin{align*}
dh_j^\alpha=h_j^\alpha B_j^\alpha-a_jh_j^\alpha.
\end{align*}
They are bounded in $W^{1,2}$ and converge almost everywhere, after a subsequence, to $h^\alpha$. It follows from \eqref{Equation: critical curvature strong convergence} that $F_{B_j^\alpha}\to F_{B^\alpha}$ strongly in $L^2$. The Coulomb and normal boundary conditions give the Hodge estimate
\begin{align*}
\|B_j^\alpha-B^\alpha\|_{W^{1,2}}
&\le C\big(\|F_{B_j^\alpha}-F_{B^\alpha}\|_{L^2}
+\|B_j^\alpha-B^\alpha\|_{L^2}\big)\\
&\quad+C\big(\|B_j^\alpha\|_{L^4}+\|B^\alpha\|_{L^4}\big)
\|B_j^\alpha-B^\alpha\|_{W^{1,2}}.
\end{align*}
Absorbing the last term proves strong $W^{1,2}$ convergence on each ball.

For patching we need uniform convergence of the transition functions, which does not follow from $W^{2,2}$ convergence alone. They satisfy
\begin{align*}
dt_j^{\alpha\beta}
=t_j^{\alpha\beta}B_j^\beta-B_j^\alpha t_j^{\alpha\beta}.
\end{align*}
The Sobolev--Lorentz embedding $W^{1,2}\hookrightarrow L^{4,2}$ and the preceding strong convergence first give an almost everywhere convergent subsequence of these bounded maps, and then strong convergence of their first derivatives in $L^{4,2}$. Applying $d^*$ and using both Coulomb conditions expresses $\Delta t_j^{\alpha\beta}$ as products of these first derivatives and the connection forms. The products converge strongly in $L^{2,1}$. Interior Lorentz elliptic estimates, together with $W^{2,(2,1)}\hookrightarrow C^0$, give uniform convergence on smaller overlaps, as well as strong $W^{2,2}$ convergence, to transition functions $t^{\alpha\beta}$. See also \cite[proof of Theorem V.5]{riviere-variations}.

For a partition of unity subordinate to smaller balls, set
\begin{align*}
T_j^\alpha:=\sum_\beta\rho_\beta\,t_j^{\alpha\beta}t^{\beta\alpha}.
\end{align*}
Then $T_j^\alpha\to I$ uniformly and strongly in $W^{2,2}$, and the cocycle identities give
\begin{align*}
T_j^\alpha t^{\alpha\gamma}=t_j^{\alpha\gamma}T_j^\gamma.
\end{align*}
Tubular projection to $G$, with an invariant matrix inner product, preserves this compatibility and convergence. The resulting bundle identifications patch the local strong $W^{1,2}$ convergence.

The equation $dh^\alpha=h^\alpha B^\alpha-A_0h^\alpha$ gives $W^{1,4}$ and then $W^{2,2}$ regularity. The same Lorentz argument, using $d^*B^\alpha=0$ and smoothness of $A_0$, gives continuity. These gauges identify the limiting connection with $A_0$ on $P$. The resulting global $W^{2,2}\cap C^0$ gauges can be approximated uniformly and strongly in $W^{2,2}$ by smooth gauges. The embeddings above make the gauge action continuous, so $A_j\to A_0$ strongly in $W^{1,2}$ in smooth gauges. Applying \eqref{Equation: critical Coulomb slice} gives the required representatives, smooth by elliptic regularity.

Finally, approximate a $W^{1,2}$ connection $A$ strongly in $W^{1,2}$ by smooth connections $A_k$. Their curvatures converge strongly in $L^2$, and
\begin{align*}
|d_2(A_k,A_0)-d_2(A,A_0)|\le\|A_k-A\|_{L^2}.
\end{align*}
Apply the smooth assertion with target norm $\delta/2$. The gauge equation bounds the resulting $h_k$ in $W^{2,2}$, as in \eqref{Equation: critical Coulomb slice}; passage to a weak limit preserves the Coulomb condition and gives the required bound by lower semicontinuity.
\end{proof}
\end{lemma}

\subsection{A Sobolev topology in arbitrary dimension}\label{subsec: Sobolev topology arbitrary dimension}
The construction extends to $W^{1,2}\cap L^{n-1}$ in every dimension $n\ge5$.
\begin{proposition}
\label{Proposition: all dimensional Sobolev epiperimetric}
Let $n\ge5$, set $d:=n-1$, and let $P\to\mathbb S^d$ be a smooth
principal bundle with compact matrix structure group. Let $A_0$ be a smooth
Yang--Mills connection on $P$. There exist $\delta,\varepsilon>0$ and
$\gamma\in[0,1)$ with the following property. If
\begin{align*}
 a\in W^{1,2}(\mathbb S^d,T^*\mathbb S^d\otimes\mathfrak g_P)
       \cap L^d(\mathbb S^d,T^*\mathbb S^d\otimes\mathfrak g_P),
 \qquad d_{A_0}^*a=0,\qquad
 \|a\|_{W^{1,2}}+\|a\|_{L^d}<\delta,
\end{align*}
then $A_0+a$ admits a competitor on the radial pullback bundle over
$\mathbb B^n\smallsetminus\{0\}$ satisfying
\begin{align*}
 \mathscr Y_{\mathbb B^n}(\hat A\,;\pi^*A_0)
 \le(1-\varepsilon|E|^\gamma)E,
 \qquad E:=\mathscr Y_{\mathbb B^n}(\pi^*(A_0+a)\,;\pi^*A_0).
\end{align*}
The competitor has trace $A_0+a$, finite energy, and
$W^{1,2}\cap L^4$ connection forms in smooth angular trivializations.
If the kernel at $A_0$ is integrable, one can take $\gamma=0$.
\begin{proof}
Retain $\mathcal E,Q,K,\Upsilon$ from the smooth construction. Lemma \ref{Lemma: Lyapunov-Schmidt reduction}-(ii) extends to $W^{1,2}$ directions by density. We claim
\begin{align}\label{Equation: all dimensional Sobolev Hessian estimate}
 \big|D^2\mathcal E(z)[v,v]-D^2\mathcal E(0)[v,v]\big|
 \le C\big(\|z\|_{L^d}+\|z\|_{L^d}^2\big)\|v\|_{W^{1,2}}^2.
\end{align}
For smooth $z,v$, the terms are bounded by the integrals of $|z||v||\nabla_{A_0}v|$ and $|z|^2|v|^2$, except for the signed term
\begin{align*}
 \int_{\mathbb S^d}\langle d_{A_0}z,[v\wedge v]\rangle
 &=\int_{\mathbb S^d}\langle z,d_{A_0}^*[v\wedge v]\rangle,\\
 |d_{A_0}^*[v\wedge v]|&\le C|v||\nabla_{A_0}v|.
\end{align*}
H\"older's inequality and $W^{1,2}\hookrightarrow L^{2d/(d-2)}$ prove the claim. Strong approximation in $W^{1,2}\cap L^d$ extends it to the required forms; since $d\ge4$, the approximation is also strong in $L^4$ and preserves the energy expansions.

Decompose $a=b+\Upsilon(b)+w_++w_-$ as before. The finite-rank smooth projections onto $K$ and the negative eigenspace preserve smallness in $L^d$, so $w_+$ is also small. Integrating \eqref{Equation: all dimensional Sobolev Hessian estimate} gives \eqref{Equation: reduced Taylor expansion} with remainder $C\delta\|w\|_{W^{1,2}}^2$.

The paths used in \eqref{Equation: kernel flow estimates} and \eqref{Equation: quadratic decrease} stay in this neighborhood: the kernel and negative components are smooth, and the positive component is multiplied by a bounded scalar. They are uniformly bounded in $W^{1,2}\cap L^d$, with radial derivative bounded in $L^2$. The cone weights $\rho^{n-5}$ and $\rho^{n-3}$ are integrable, giving the asserted $W^{1,2}\cap L^4$ regularity. The same estimates in \eqref{equation: starting estimate} therefore give \eqref{Equation: total competitor decrease}. The argument following that equation proves the assertion, with $\gamma=0$ in the integrable case.
\end{proof}
\end{proposition}

\begin{remark}
A spherical $W^{2,2}\cap W^{1,4}$ gauge extends radially in the same class for $n>4$, so the competitor may be gauged back to the original trace. The relative gauges in the next lemma belong to $W^{2,q}$, $q>(n-1)/2$, and hence have this regularity.
\end{remark}

The following criterion gives the required neighborhood from tangential curvature control.
\begin{lemma}
\label{Lemma: curvature controlled energy distance}
Let $d\ge4$, let $P\to\mathbb S^d$ be a smooth principal bundle with compact matrix structure group,
and let $a_0$ be a smooth connection on $P$. Fix $p>d/2$, $M<\infty$ and
$q\in(d/2,\min\{p,d\})$. For every $\delta>0$ there is $\eta>0$ such that
every $W^{1,p}$ connection $a$ on $P$ satisfying
\begin{align*}
 \|F_a\|_{L^p(\mathbb S^d)}\le M,\qquad d_2(a,a_0)<\eta,\qquad
 \operatorname{YM}_{\mathbb S^d}(a)\le\operatorname{YM}_{\mathbb S^d}(a_0)+\eta
\end{align*}
admits a $W^{2,q}$ gauge $u$ with
\begin{align*}
 d_{a_0}^*(a^u-a_0)=0,\qquad
 \|a^u-a_0\|_{W^{1,q}(\mathbb S^d)}<\delta.
\end{align*}
Here $d_2$ is defined in \eqref{Equation: weak gauge distance}. In particular,
the conclusion gives smallness in $W^{1,2}\cap L^d$.
\begin{proof}
Suppose $\|F_{a_j}\|_{L^p}\le M$, $d_2(a_j,a_0)\to0$ and $\limsup_j\operatorname{YM}(a_j)\le\operatorname{YM}(a_0)$. Choose $W^{1,2}$ gauges giving representatives $b_j\to a_0$ strongly in $L^2$. Since $W^{1,p}\subset L^4$, the covariance and energy argument in Lemma \ref{Lemma: critical energy distance} gives $F_{b_j}\to F_{a_0}$ strongly in $L^2$. Interpolation yields
\begin{align*}
 F_{b_j}\rightarrow F_{a_0}\qquad\text{strongly in }L^q.
\end{align*}

Since $p>d/2$, the curvature is uniformly small in $L^{d/2}$ on sufficiently small fixed balls. Local Coulomb gauges with normal boundary condition for the original $W^{1,p}$ connections $a_j$ give forms $B_j^\alpha$ uniformly bounded in $W^{1,p}$ and small in $L^d$, converging weakly in $W^{1,p}$ and strongly in $L^q$ to $B^\alpha$. The gauges from $b_j$ to $B_j^\alpha$ satisfy
\begin{align*}
 dh_j^\alpha=h_j^\alpha B_j^\alpha-b_jh_j^\alpha.
\end{align*}
Their $W^{1,2}$ bound gives an almost everywhere convergent subsequence; curvature covariance then gives $F_{B_j^\alpha}\to F_{B^\alpha}$ strongly in $L^q$. The Hodge estimate gives
\begin{align*}
 \|B_j^\alpha-B^\alpha\|_{W^{1,q}}
 &\le C\big(\|F_{B_j^\alpha}-F_{B^\alpha}\|_{L^q}
                   +\|B_j^\alpha-B^\alpha\|_{L^q}\big)\\
 &\quad+C\big(\|B_j^\alpha\|_{L^d}+\|B^\alpha\|_{L^d}\big)
       \|B_j^\alpha-B^\alpha\|_{W^{1,q}}.
\end{align*}
The embedding $W^{1,q}\hookrightarrow L^{dq/(d-q)}$ and smallness in $L^d$ allow absorption, proving strong $W^{1,q}$ convergence.

Decreasing $p$ with $q<p$ if necessary, assume $p<d$. The transition equation
\begin{align*}
 dt_j^{\alpha\beta}=t_j^{\alpha\beta}B_j^\beta-B_j^\alpha t_j^{\alpha\beta}
\end{align*}
bounds the transition functions in $W^{1,dp/(d-p)}$, hence gives uniform convergence on smaller overlaps because $dp/(d-p)>d$. The same equation gives strong convergence of their derivatives in $L^{dq/(d-q)}$; differentiation gives strong $W^{2,q}$ convergence, since $dq/(d-q)>2q$. The patching construction of Lemma \ref{Lemma: critical energy distance} therefore applies. The limiting gauges satisfy $dh^\alpha=h^\alpha B^\alpha-a_0h^\alpha$, which gives $W^{2,q}\subset C^0$ regularity and identifies the limit with $a_0$ on $P$. We obtain global strong $W^{1,q}$ convergence. The relative Coulomb slice theorem, \cite[Theorem 14]{FeehanMaridakis}, completes the proof.
\end{proof}
\end{lemma}

\section{Proof of the uniqueness of tangent cones with isolated singularities} \label{sec: proof of uniqueness}
Let $A$, $y$ and $\varphi$ be as in Theorem \ref{Theorem: uniqueness of tangent connections with isolated singularities}. Restricting to a trivializing ball about $y$, we identify $A$ with its local connection form. After translating, we may assume $y=0$. We write $\nu(x):=x/|x|$ for the radial unit vector field. Set
\begin{align*}
    m:=n-4,\qquad \theta:=\operatorname{YM}_{\mathbb{B}^n}(\varphi),\qquad
    a(r):=\iota_{\mathbb{S}^{n-1}}^*A_{0,r},\qquad
    a_0:=\iota_{\mathbb{S}^{n-1}}^*\varphi.
\end{align*}
Here
\begin{align*}
    \Theta(r,y;A):=r^{4-n}\int_{B_r(y)}|F_A|^2\,dx,
    \qquad \Theta(y;A):=\lim_{r\downarrow0}\Theta(r,y;A).
\end{align*}
The smooth link $a_0$ is Yang--Mills. For minimizers, this follows by passing compactly supported variations to the strong annular tangent limit; for $\omega$-ASD connections, the limiting equation has the constant closed form $\omega(0)$. The strong curvature convergence in the assumptions and the almost monotonicity formula imply that $\Theta(0;A)=\theta$. By the same formula, there exist $K,\alpha_0>0$ such that
\begin{align*}
    E(r):=\Theta(r,0;A)-\theta+Kr^{\alpha_0}
\end{align*}
is nonnegative and non-decreasing for all sufficiently small $r$, tends to zero as $r\downarrow0$, and controls the radial curvature energy. We shall increase $K$ and decrease the range of radii when necessary.

We first record the compactness argument needed to apply the epiperimetric inequality at every sufficiently small scale. We use the distance $d_2$ defined in \eqref{Equation: weak gauge distance}. Given $\delta>0$, there exist $\eta,r_*>0$ such that, whenever $0<r<r_*$, the trace $a(r)$ is smooth modulo gauge and
\begin{align}\label{Equation: annular compactness assumptions}
    E(2r)<\eta,\qquad d_2(a(r),a_0)<\eta,
\end{align}
the rescaling $A_{0,r}$ is smooth modulo gauge on $B_{3/2}\smallsetminus\overline{B}_{1/2}$ and lies within $\delta$ of $\varphi$ there in $C^{2,\beta}$, for a fixed $\beta\in(0,1)$.

Indeed, suppose this failed for $r_j\downarrow0$ with both quantities in \eqref{Equation: annular compactness assumptions} tending to zero. Choose gauges on $\mathbb{S}^{n-1}$ in which $a(r_j)\to a_0$ in $L^2$, and extend them by radial parallel transport. In this gauge, writing $a_j(t)=\iota_{\mathbb{S}^{n-1}}^*(A_{0,r_j})_{0,t}$, we have
\begin{align}\label{Equation: radial gauge derivative}
    \partial_t a_j(t)=t\,(F_{A_{0,r_j}}\res\nu)(t\,\cdot\,).
\end{align}
Radial parallel transport is defined along almost every ray. The identity follows first for smooth connections and then for $W^{1,2}\cap L^4$ connections by approximation on annuli. Only $L^2$ regularity of the resulting connection forms is needed here.
The almost monotonicity formula and Cauchy--Schwarz show that $a_j(t)\to a_0$ in $L^2$, uniformly on every compact interval in $(0,2)$. Thus $A_{0,r_j}\to\varphi$ strongly in $L^2$ on compact annuli in this gauge. It follows that
\begin{align*}
    A_{0,r_j}\wedge A_{0,r_j}\rightarrow\varphi\wedge\varphi
\end{align*}
in $L^1$ there, so that the curvatures converge distributionally, and hence weakly in $L^2$, to $F_\varphi$. Since
\begin{align*}
    \operatorname{YM}_{B_2}(A_{0,r_j})\rightarrow 2^{n-4}\theta
    =\operatorname{YM}_{B_2}(\varphi),
\end{align*}
lower semicontinuity on an exhaustion of $B_2\smallsetminus\{0\}$ gives strong curvature convergence on compact annuli. On a slightly larger annulus, choose $s_0>0$ so that the scaled energy of $\varphi$ on every $s_0$-ball is small. Strong convergence gives the same bound for $A_{0,r_j}$, uniformly in the centers. Monotonicity at all centers propagates this bound to smaller radii: use Proposition \ref{Proposition: almost monotonicity formula} for minimizers and Lemma \ref{Lemma: torsion monotonicity} for $\omega$-ASD connections. The rescaled forms $\omega(r_j\,\cdot\,)$ have uniformly bounded derivatives. Minimizers are stationary by domain variations, so \cite{TaoTian} gives removability and $\varepsilon$-regularity. For $\omega$-ASD connections, the curvature estimate in \cite[Section 6]{wentworth-chen} and the local gauge construction in \cite[Appendix C]{SmithUhlenbeck} give smooth extension across the singular set.

Write $A_j$ for the radial representatives and $C_j$ for smoothly convergent local representatives. Their transition gauges satisfy
\begin{align*}
 du_j=u_jC_j-A_ju_j.
\end{align*}
They are bounded in $W^{1,2}$ and, after a subsequence, converge almost everywhere. The right-hand side converges strongly in $L^2$, so the limiting gauge identifies the smooth local limit with $\varphi$ and is smooth by the same equation. These identifications are compatible on overlaps. Patching the smoothly converging transitions gives smooth convergence modulo gauge on $B_{3/2}\smallsetminus\overline{B}_{1/2}$, a contradiction.

Let $\varepsilon,\delta>0$ and $\gamma\in[0,1)$ be the constants in Theorem \ref{Theorem: (log)-epiperimetric inequality for Yang--Mills functional} at $a_0$. Consider any interval of radii on which the traces are smooth and lie, modulo gauge, in this $C^{2,\beta}$ neighbourhood. Write $\tilde a(r)$ for the homogeneous extension of $a(r)$ and set
\begin{align*}
    e(r)&:=\Theta(r,0;A)-\theta,\\
    z(r)&:=\operatorname{YM}_{\mathbb{B}^n}(\tilde a(r))-\theta,\\
    R(r)&:=\int_{\mathbb{S}^{n-1}}|F_{A_{0,r}}\res\nu|^2\,d\mathscr{H}^{n-1}.
\end{align*}
By coarea and the orthogonal decomposition of the curvature into its tangential and radial parts,
\begin{align}\label{Equation: excess derivative identity}
    r e'(r)=m\big(z(r)-e(r)\big)+R(r)
\end{align}
for almost every $r$. Almost minimality, applied first to the homogeneous extension and then to the epiperimetric competitor, gives
\begin{align}\label{Equation: corrected epiperimetric comparisons}
    e(r)\le z(r)+C_0r^{\alpha_0},\qquad
    e(r)\le z(r)-\varepsilon|z(r)|^\gamma z(r)+C_0r^{\alpha_0}.
\end{align}
The boundary gauges used to apply the epiperimetric inequality extend radially over the ball. Indeed, a $W^{1,2}$ gauge taking a $W^{1,2}\cap L^4$ trace to a smooth representative belongs to $W^{2,2}\cap W^{1,4}$: the transformation law first gives its derivative in $L^4$, and differentiation then gives its second derivatives in $L^2$. Since $n>4$, its radial extension preserves the admissibility of the competitors.

If $e(r)\ge2C_0r^{\alpha_0}$, then $z(r)\ge e(r)/2>0$. The second inequality in \eqref{Equation: corrected epiperimetric comparisons} and \eqref{Equation: excess derivative identity} therefore give
\begin{align*}
    r e'(r)\ge c\,e(r)^{1+\gamma}-Cr^{\alpha_0}+R(r).
\end{align*}
If $e(r)<2C_0r^{\alpha_0}$, the first inequality gives the same estimate with $e(r)$ replaced by its positive part and a larger $C$. Increasing $K$, decreasing $r_*$ and reducing $c$ if necessary, we obtain
\begin{align}\label{Equation: corrected excess differential inequality}
    r E'(r)\ge c\,E(r)^{1+\gamma}+R(r).
\end{align}
Here and below the positive constants may change from line to line.

Choose radial gauge on the interval under consideration. By \eqref{Equation: radial gauge derivative},
\begin{align*}
    \|a'(r)\|_{L^2}=\frac{\sqrt{R(r)}}{r}
    \le\frac{E'(r)}{\sqrt c\,E(r)^{(1+\gamma)/2}}
\end{align*}
where $E(r)>0$. Integrating, and treating intervals where $E=0$ by continuity, yields
\begin{align}\label{Equation: radial length bound}
    \|a(r)-a(s)\|_{L^2}\le
    C\big(E(r)^{(1-\gamma)/2}-E(s)^{(1-\gamma)/2}\big),\qquad 0<s<r.
\end{align}
This is the estimate which allows us to continue the epiperimetric inequality to smaller scales.

By the strong tangent convergence and $\varepsilon$-regularity, we may choose a scale $r_0<r_*/2$ such that the traces are initially smooth and close to $a_0$ on an interval about $r_0$, $E(2r_0)<\eta$, and
\begin{align*}
    d_2(a(r_0),a_0)+C E(r_0)^{(1-\gamma)/2}<\eta/2.
\end{align*}
On every interval of smaller radii reached so far, \eqref{Equation: radial length bound} keeps $d_2(a(r),a_0)<\eta/2$. A positive first stopping scale is impossible: applying the annular compactness argument at radii just above it gives smoothness and $C^{2,\beta}$ closeness on an annulus extending below that scale. Thus the epiperimetric inequality, \eqref{Equation: corrected excess differential inequality}, and \eqref{Equation: radial length bound} hold for every sufficiently small radius. In particular, $A$ is smooth modulo gauge in a punctured neighbourhood of $0$. The annular estimates are uniform in $r$, so
\begin{align*}
 \sup_{B_{3/2}\smallsetminus\overline B_{1/2}}|F_{A_{0,r}}|\le C.
\end{align*}
Rescaling gives $|F_A(x)|\le C|x|^{-2}$ for all sufficiently small nonzero $x$.

The quantitative $C^k$ estimates follow from the next lemma.

\begin{lemma}\label{Lemma: quantitative gauge alignment}
Let $A_0$ be a smooth connection on $\mathbb S^{n-1}$ and realize the compact group $G$ as a subgroup of a unitary group. There exist $\delta,C>0$ such that, if $a$ is a smooth $\mathfrak g$-valued $1$-form satisfying
\begin{align*}
d_{A_0}^*a=0,\qquad \|a\|_{L^\infty}<\delta,
\end{align*}
then
\begin{align*}
\|a\|_{L^2}\le C\inf_u\|(A_0+a)^u-A_0\|_{L^2},
\end{align*}
where the infimum is over $W^{1,2}$ gauge transformations.
\begin{proof}
Fix such a gauge $u$, and write $\eta:=(A_0+a)^u-A_0$. Regarding $u$ as an endomorphism-valued function, let $D_0u:=du+A_0u-uA_0$. Then
\begin{align*}
D_0u+au=u\eta.
\end{align*}
Let $v$ be the $L^2$-projection of $u$ onto $\ker D_0$ and set $\xi:=u-v$. The Poincar\'e inequality on the orthogonal complement of $\ker D_0$ gives
\begin{align*}
\|\xi\|_{W^{1,2}}\le C\|D_0u\|_{L^2}
\le C(\|\eta\|_{L^2}+\|a\|_{L^2}).
\end{align*}
We may assume $\|\eta\|_{L^2}\le\delta$, since otherwise the conclusion follows directly from $\|a\|_{L^\infty}<\delta$ after increasing $C$. The endomorphism $v^*v-I$ is parallel, and $u^*u=I$. Since the finite-dimensional projection onto $\ker D_0$ bounds $\|v\|_{L^\infty}$ uniformly, it follows that
\begin{align*}
\|v^*v-I\|_{L^\infty}
\le C\|v^*v-u^*u\|_{L^1}
\le C\|\xi\|_{L^2}\le C\delta.
\end{align*}
For small $\delta$, $v$ is therefore uniformly invertible. Moreover, $D_0^*(av)=0$, since $v$ is parallel and $d_{A_0}^*a=0$. Pairing $D_0\xi+av+a\xi=u\eta$ with $av$ and integrating gives
\begin{align*}
\|av\|_{L^2}^2
&=\langle av,u\eta-a\xi\rangle_{L^2}\\
&\le\|av\|_{L^2}\big(\|\eta\|_{L^2}+\|a\|_{L^\infty}\|\xi\|_{L^2}\big).
\end{align*}
The uniform invertibility of $v$ and the preceding estimate for $\xi$ imply
\begin{align*}
\|a\|_{L^2}\le C\|\eta\|_{L^2}+C\delta\|a\|_{L^2}.
\end{align*}
Absorbing the last term and taking the infimum over $u$ proves the claim.
\end{proof}
\end{lemma}

If $\gamma>0$, integration of \eqref{Equation: corrected excess differential inequality} gives
\begin{align}\label{Equation: logarithmic decay of the energy density}
    E(r)\le\big(E(r_0)^{-\gamma}+c\gamma\log(r_0/r)\big)^{-1/\gamma},
    \qquad 0<r<r_0,
\end{align}
with the zero-excess case understood separately. Proposition \ref{Proposition: Dini decrease implies uniqueness of tangent cone} applies to this decay for every $0<\gamma<1$ and gives a unique radial limit. For $0<s<r$, \eqref{Equation: radial length bound} gives
\begin{align*}
    d_2(a(r),a(s))\le C E(r)^{(1-\gamma)/2}.
\end{align*}
Along the original tangent sequence, $a(s_i)$ converges smoothly to $a_0$ modulo gauge. Taking $s=s_i\downarrow0$ and using the triangle inequality, we obtain
\begin{align*}
    d_2(a(r),a_0)\le C E(r)^{(1-\gamma)/2}
    \le C|\log r|^{-(1-\gamma)/(2\gamma)}.
\end{align*}

For completeness, choose smooth representatives $a_0+b(r)$ in the local relative Coulomb slice at $a_0$. The annular $\varepsilon$-regularity and elliptic estimates, followed by the local slice theorem, give uniform $C^l$ bounds for these representatives for every $l$. Their $L^\infty$ distance from $a_0$ is small, so Lemma \ref{Lemma: quantitative gauge alignment} gives
\begin{align*}
    \|b(r)\|_{L^2}\le C d_2(a(r),a_0).
\end{align*}
Interpolation between $L^2$ and the uniformly bounded $H^{2k+n+3}$ norm, followed by Sobolev embedding on $\mathbb{S}^{n-1}$, yields
\begin{align*}
    \|b(r)\|_{C^k}\le C_k\|b(r)\|_{L^2}^{1/2}
    \le C_k|\log r|^{-(1-\gamma)/(4\gamma)}.
\end{align*}
Thus a single positive decay exponent works for every $k$.
If $\gamma=0$, instead we obtain
\begin{align*}
    E(r)\le E(r_0)(r/r_0)^c,\qquad
    d_2(a(r),a_0)\le Cr^{c/2}.
\end{align*}
Proposition \ref{Proposition: Dini decrease implies uniqueness of tangent cone} also applies to this power decay, and the same argument gives $\|b(r)\|_{C^k}\le C_kr^{c/4}$.
\subsection{Almost minimizers in dimension $5$}
\label{subsec: uniqueness almost minimizers five dimensions}
\begin{proof}[Proof of Theorem \ref{Theorem: uniqueness for almost minimizers}-(i)]
Restricting to a ball about $y$, translating and rescaling, we work on
$Q=\pi^*P\to\mathbb B^5\smallsetminus\{0\}$ with $y=0$.
We use the notation above with $n=5$ and $m=1$. In particular,
\begin{align*}
    \theta&:=\operatorname{YM}_{\mathbb B^5}(\varphi)
             =\operatorname{YM}_{\mathbb S^4}(a_0),\qquad
    a(r):=\iota_{\mathbb S^4}^*A_{0,r},\\
    e(r)&:=\Theta(r,0;A)-\theta,\qquad
    z(r):=\operatorname{YM}_{\mathbb S^4}(a(r))-\theta,\qquad
    R(r):=\int_{\mathbb S^4}|F_{A_{0,r}}\res\nu|^2\,d\mathscr H^4.
\end{align*}
The spherical energies are defined for almost every radius, using the
original $W^{1,2}$ traces. The homogeneous extension of each such trace
is an admissible comparison. Thus coarea and
\eqref{Equation: centered comparison all dimensions} give
\begin{align}\label{Equation: five dimensional cone and coarea}
    re'(r)=z(r)-e(r)+R(r),\qquad
    e(r)\le z(r)+C_0r^{\alpha_0}.
\end{align}
It follows that $\Theta(r,0;A)+(C_0/\alpha_0)r^{\alpha_0}$ is
non-decreasing and controls $\int R(r)\,dr/r$. Its limit at $0$ is finite.
On the annulus in the hypotheses, strong curvature convergence identifies
this limit with $\theta$, since the annular energies converge both to
$(t-s)\Theta(0;A)$ and to $(t-s)\theta$.

To prove that $a_0$ is Yang--Mills, write $A_i=A_{0,\rho_i}^{g_i}$ and take a smooth variation $\eta$ compactly supported in the convergence annulus. For fixed $u\in\mathbb R$,
\begin{align*}
 F_{A_i+u\eta}=F_{A_i}+u\,d_{A_i}\eta+u^2\eta\wedge\eta
\end{align*}
converges strongly in $L^2$ to $F_{\varphi+u\eta}$. The pulled-back variation is bounded and $W^{1,2}$, hence admissible. Apply centered almost minimality on an enclosing ball and cancel the unchanged energy. The rescaled error is $O(\rho_i^{\alpha_0})$, so passage to the limit proves that $\varphi$, and hence $a_0$, is Yang--Mills.

We extend $d_2$ to $L^2$ representatives and use radial gauge. Solve parallel transport with identity initial value on a compact annulus. Differentiating in the angular variables expresses the derivative of the gauge as an integral of conjugates of $\partial_\vartheta A_r\in L^2$. Thus the gauge belongs to $W^{1,2}$, with radial derivative in $L^4$. The transformed angular connection is locally absolutely continuous into $L^2$, with
\begin{align*}
 \|a'(r)\|_{L^2}=\frac{\sqrt{R(r)}}r.
\end{align*}
Curvature covariance holds distributionally by the Sobolev product rule and
$d(g^{-1}dg)=-(g^{-1}dg)\wedge(g^{-1}dg)$ for bounded $W^{1,2}$ group-valued maps. The spherical gauges are $W^{1,2}$ for almost every radius, so $d_2(a(r),a_0)$ is continuous and Lemma \ref{Lemma: critical energy distance} applies to the original traces.

Let $\delta,\varepsilon>0$ and $\gamma\in[0,1)$ be the constants in Theorem \ref{Theorem: weak epiperimetric five dimensions}. Lemma \ref{Lemma: critical energy distance} gives $\eta>0$ such that
\begin{align*}
 d_2(b,a_0)<\eta,\qquad
 \operatorname{YM}_{\mathbb S^4}(b)<\operatorname{YM}_{\mathbb S^4}(a_0)+\eta
\end{align*}
imply $\|b^g-a_0\|_{W^{1,2}}<\delta$ in relative Coulomb gauge. Take $e(r)<\eta/2$ and consider an interval where $d_2(a(r),a_0)<\eta$. If $z(r)\ge\eta$, \eqref{Equation: five dimensional cone and coarea} gives $re'(r)\ge\eta/2+R(r)$. Otherwise the epiperimetric comparison gives
\begin{align*}
 e(r)\le z(r)-\varepsilon|z(r)|^\gamma z(r)+C_0r^{\alpha_0}.
\end{align*}
The boundary gauge belongs to $W^{2,2}\cap W^{1,4}$, and its radial extension preserves admissibility. The argument leading to \eqref{Equation: corrected excess differential inequality} therefore applies in both cases. For $K>2C_0/\alpha_0$ sufficiently large and $E(r)=e(r)+Kr^{\alpha_0}$,
\begin{align}\label{Equation: weak five dimensional excess ODE}
 rE'(r)\ge cE(r)^{1+\gamma}+R(r).
\end{align}
When $\gamma=0$, decrease $c$ in terms of $\alpha_0$. Corrected monotonicity gives $e(r)\ge-(C_0/\alpha_0)r^{\alpha_0}$, hence $E(r)>0$ and $Kr^{\alpha_0}\le2E(r)$.

The radial length argument leading to \eqref{Equation: radial length bound}
now applies on every interval where $d_2(a(r),a_0)<\eta$. With
$q:=(1-\gamma)/2$, it gives
\begin{align}\label{Equation: weak five dimensional radial length}
    d_2(a(r),a(s))\le\int_s^r\|a'(u)\|_{L^2}\,du
       \le C\big(E(r)^q-E(s)^q\big),\qquad 0<s<r.
\end{align}
By Fubini, the tangent convergence supplies a sufficiently small radius $r_0$ such that
\begin{align*}
 d_2(a(r_0),a_0)+CE(r_0)^q<\eta/2.
\end{align*}
On the maximal interval extending downwards from $r_0$ where the distance is below $\eta$, \eqref{Equation: weak five dimensional radial length} keeps it below $\eta/2$. Continuity excludes a positive stopping radius. Thus \eqref{Equation: weak five dimensional excess ODE} holds at almost every sufficiently small radius.

Integration gives \eqref{Equation: logarithmic decay of the energy density}
when $\gamma>0$, and $E(r)\le Cr^c$ when $\gamma=0$.
Proposition \ref{Proposition: Dini decrease implies uniqueness of tangent cone}
gives a unique radial limit, identified with $a_0$ modulo gauge by the
original tangent sequence. Letting $s$ tend to zero along that sequence in
\eqref{Equation: weak five dimensional radial length}, we obtain
\begin{align*}
    d_2(a(r),a_0)\le CE(r)^q.
\end{align*}
In particular, the logarithmic exponent can be taken to be
$\beta=(1-\gamma)/(2\gamma)$ if $\gamma>0$.

It remains to check strong curvature convergence. On a fixed annulus
$U=B_\rho\smallsetminus\overline{B_\sigma}$, choose a gauge on a middle
sphere which realizes the preceding distance up to an arbitrarily small
error, and extend it radially. The length estimate gives, in this gauge,
\begin{align*}
    \|b\|_{L^2(U)}\le C_UE(\rho r)^q,
    \qquad b:=A_{0,r}-\varphi.
\end{align*}
Since $d_\varphi^*F_\varphi=0$ and $F_\varphi\res\nu=0$, testing against
radial cutoffs gives $\int_U\langle d_\varphi b,F_\varphi\rangle=0$.
This is valid for the $L^2$ radial representative because
$d_\varphi b=F_{A_{0,r}}-F_\varphi-b\wedge b\in L^1(U)$. Consequently,
\begin{align*}
    \|F_{A_{0,r}}-F_\varphi\|_{L^2(U)}^2
    &\le\big|\operatorname{YM}_U(A_{0,r})-
                    \operatorname{YM}_U(\varphi)\big|
          +C_U\|b\|_{L^2(U)}^2\\
    &\le C_U\big(E(\rho r)+E(\rho r)^{1-\gamma}\big).
\end{align*}
Here we used
$\operatorname{YM}_U(A_{0,r})-\operatorname{YM}_U(\varphi)
=\rho e(\rho r)-\sigma e(\sigma r)$, monotonicity of $E$, and
$Kr^{\alpha_0}\le2E(r)$. This proves the annular estimates with the same
logarithmic, or power, rate.

For Remark \ref{Remark: trivial bundle comparison class}, the cone and epiperimetric competitors and the radial boundary gauges have the required global Sobolev regularity in dimension $5$. The compactly supported variations are bounded $W^{1,2}$ forms, so adding them to $A$ also preserves admissibility.

\end{proof}

\begin{remark}\label{Remark: instanton links integrable}
    For a self-dual or anti-self-dual link on the round $\mathbb S^4$, the
    integrability assumption is automatic. In the anti-self-dual case,
    the kernel of the Yang--Mills Hessian on the Coulomb slice is
    $\ker d_{a_0}^*\cap\ker d_{a_0}^+$. The Weitzenb\"ock formula gives
    \begin{align*}
        2d_{a_0}^+d_{a_0}^{+,*}=\nabla_{a_0}^*\nabla_{a_0}+4,
    \end{align*}
    since the scalar curvature is $12$, $W^+=0$ and $F_{a_0}^+=0$;
    see \cite{freed-uhlenbeck}. Thus the obstruction space
    $\ker d_{a_0}^{+,*}$ vanishes, and the implicit function theorem for
    the anti-self-dual equation in relative Coulomb gauge integrates
    every Jacobi field. For reducible connections, the scalar gauge
    equation is restricted to the orthogonal complement of the
    stabilizer. The self-dual case follows by reversing orientation.
    Hence Theorem \ref{Theorem: uniqueness for almost minimizers}-(i) gives
    power decay for instanton tangent links.
\end{remark}

\subsection{Almost minimizers with controlled tangential curvature}
\label{subsec: uniqueness almost minimizers all dimensions}

\begin{lemma}[Existence of a smooth strong tangent]
\label{Lemma: automatic strong tangent from tangential curvature}
Under the assumptions of Theorem \ref{Theorem: uniqueness for almost minimizers}-(ii), $A$ has a strong tangent at $y$, $\varphi=\pi^*a_0$, with smooth Yang--Mills link on $P$, modulo gauge.
\begin{proof}
Restricting to a ball about $y$, translating and rescaling, we work on
$Q=\pi^*P\to\mathbb B^n\smallsetminus\{0\}$ with $y=0$.
Write $d=n-1$, $m=n-4$, $a(r)=\iota_{\mathbb S^d}^*A_{0,r}$ and
\begin{align*}
 \Theta(r):=r^{-m}\operatorname{YM}_{B_r}(A),\qquad
 R(r):=\int_{\mathbb S^d}|F_{A_{0,r}}\res\nu|^2.
\end{align*}
Cone comparison and coarea give
\begin{align*}
 r\Theta'(r)&=\operatorname{YM}_{\mathbb S^d}(a(r))-m\Theta(r)+R(r),\\
 \Theta(r)&\le m^{-1}\operatorname{YM}_{\mathbb S^d}(a(r))+C_0r^{\alpha_0}.
\end{align*}
Thus $\Theta(r)+(mC_0/\alpha_0)r^{\alpha_0}$ is non-decreasing with finite limit $\Theta_0$ at $0$. For $K\ge mC_0/\alpha_0$, the nonnegative excess $E_0(r)=\Theta(r)-\Theta_0+Kr^{\alpha_0}$ tends to zero and satisfies
\begin{align}\label{Equation: radial tail all dimensions}
 \int_s^r\frac{R(t)}t\,dt\le E_0(r)-E_0(s).
\end{align}

Choose admissible radii $r_i\downarrow0$ and, decreasing $p$ if necessary, assume $d/2<p<d$. Uhlenbeck compactness \cite{uhlenbeck-connections,wehrheim} gives a subsequence and spherical gauges in which the traces, denoted by $a_i$, converge weakly in $W^{1,p}$ to $a_0$ and strongly in $L^s$ for some $s>d$, hence in $L^2\cap L^4$. The local transition functions are bounded in $W^{2,p}$ and converge uniformly on smaller overlaps; patching identifies the limiting bundle with $P$. The transformation equation between the original and gauged traces gives $W^{2,2}\cap W^{1,4}$ regularity of the spherical gauges.

Set $\varphi=\pi^*a_0$ and $\theta=m^{-1}\operatorname{YM}_{\mathbb S^d}(a_0)$. Extend the gauges by radial parallel transport, as in the five-dimensional proof. On compact annuli they belong to $W^{1,2}$, and their angular representatives satisfy
\begin{align*}
 \|a_i(t)-a_i(1)\|_{L^2}
 \le\left(|\log t|\int_{\min\{1,t\}r_i}^{\max\{1,t\}r_i}\frac{R(u)}u\,du\right)^{1/2}
 \rightarrow0
\end{align*}
locally uniformly in $t>0$. Consequently $A_{0,r_i}\rightarrow\varphi$ strongly in $L^2$ on compact annuli and the curvatures converge weakly in $L^2$. Lower semicontinuity gives
\begin{align*}
 (t^m-s^m)\theta
 \le\lim_i\operatorname{YM}_{B_t\smallsetminus\overline{B_s}}(A_{0,r_i})
 =(t^m-s^m)\Theta_0,
\end{align*}
so $\theta\le\Theta_0$.

Interpolate between $a_i$ and $a_0$ on $B_1\smallsetminus B_{1-\lambda_i}$ and use $\varphi$ inside. Lemma \ref{Lemma: Luckhaus type lemma for connections} applies on the fixed bundle because $a_i-a_0$ is a global adjoint-valued one-form. Its shell energy is bounded by
\begin{align*}
 C\lambda_i\big(\|F_{a_i}\|_2^2+\|F_{a_0}\|_2^2
 +\|(a_i-a_0)\wedge(a_i-a_0)\|_2^2\big)
 +C\lambda_i^{-1}\|a_i-a_0\|_2^2.
\end{align*}
The parenthesized term is bounded. Choose $\lambda_i\downarrow0$ with $\lambda_i^{-1}\|a_i-a_0\|_2^2\rightarrow0$, and reverse the radially extended spherical gauge. Almost minimality then gives
\begin{align*}
 \Theta(r_i)\le(1-\lambda_i)^m\theta+o(1)+C_0r_i^{\alpha_0}.
\end{align*}
Hence $\Theta_0=\theta$, proving strong curvature convergence on compact annuli.

As in the five-dimensional proof, pulling back compactly supported smooth variations through the convergence gauges gives bounded $W^{1,2}$ variations of the original connection. Strong convergence and the vanishing comparison error show that $a_0$ is weakly Yang--Mills. Since $a_0\in W^{1,p}$ with $p>d/2$, \cite[Theorem 2.11.1]{FeehanMaridakis} gives a $W^{2,p}$ gauge in which $a_0$ is smooth. Its radial extension preserves the strong annular convergence.
\end{proof}
\end{lemma}

\begin{proof}[Proof of Theorem \ref{Theorem: uniqueness for almost minimizers}-(ii)]
Lemma \ref{Lemma: automatic strong tangent from tangential curvature} supplies $\varphi=\pi^*a_0$. Retain the local identification and normalization $y=0$ from its proof. Use the preceding notation with $m=n-4$, $\theta=\operatorname{YM}_{B_1}(\varphi)$ and
\begin{align*}
 e(r):=\Theta(r,0;A)-\theta,\qquad
 z(r):=m^{-1}\operatorname{YM}_{\mathbb S^{n-1}}(a(r))-\theta.
\end{align*}
Cone comparison and coarea give
\begin{align}\label{Equation: all dimensional cone and coarea}
 re'(r)=m(z(r)-e(r))+R(r),\qquad e(r)\le z(r)+C_0r^{\alpha_0}.
\end{align}
Lemma \ref{Lemma: curvature controlled energy distance} and the embedding $W^{1,q}\hookrightarrow W^{1,2}\cap L^{n-1}$ for $q>(n-1)/2$ give $\eta>0$ such that $d_2(a(r),a_0)<\eta$ and $z(r)<\eta$ place the trace in the neighborhood of Proposition \ref{Proposition: all dimensional Sobolev epiperimetric}. The transformation equation gives $W^{2,2}\cap W^{1,4}$ boundary gauges, whose radial extensions preserve admissibility for $n>4$.

The five-dimensional argument now applies with the following constants. When $e(r)<\eta/2$ and $z(r)\ge\eta$, coarea gives $re'(r)\ge m\eta/2+R(r)$; when $z(r)<\eta$, the epiperimetric comparison gives
\begin{align*}
 e(r)\le z(r)-\varepsilon|z(r)|^\gamma z(r)+C_0r^{\alpha_0}.
\end{align*}
Since $e(r)\ge-(mC_0/\alpha_0)r^{\alpha_0}$, choosing $K$ sufficiently large yields
\begin{align}\label{Equation: weak all dimensional excess ODE}
 rE'(r)\ge cE(r)^{1+\gamma}+R(r),\qquad E(r)=e(r)+Kr^{\alpha_0}>0.
\end{align}
For $\gamma=0$, decrease $c$ in terms of $\alpha_0$. The radial length estimate \eqref{Equation: radial length bound} and the first-exit argument in the five-dimensional proof extend this inequality to almost every sufficiently small radius. They give
\begin{align*}
 d_2(a(r),a_0)\le CE(r)^{(1-\gamma)/2},
\end{align*}
with $E(r)\le C|\log r|^{-1/\gamma}$ for $\gamma>0$ and $E(r)\le Cr^c$ for $\gamma=0$.

Proposition \ref{Proposition: Dini decrease implies uniqueness of tangent cone} gives an $L^2$ radial limit $a_*$. The original tangent sequence gives $d_2(a_*,a_0)=0$. Approximately aligning gauges have uniformly bounded $W^{1,2}$ norm by their transformation equation; a weakly convergent subsequence yields a gauge taking $a_*$ to $a_0$. Extending it radially gives $W^{1,2}$ annular gauges. The connection estimate and curvature calculation in the five-dimensional proof apply, with annular energy difference $\rho^m e(\rho r)-\sigma^m e(\sigma r)$. Hence, for $U=B_\rho\smallsetminus\overline{B_\sigma}$,
\begin{align*}
 \|A_{0,r}-\varphi\|_{L^2(U)}&\le C_UE(\rho r)^{(1-\gamma)/2},\\
 \|F_{A_{0,r}}-F_\varphi\|_{L^2(U)}^2
 &\le C_U\big(E(\rho r)+E(\rho r)^{1-\gamma}\big).
\end{align*}
This proves the theorem. The global Sobolev variant in Remark \ref{Remark: trivial bundle comparison class} follows since all competitors and radial boundary gauges have the required regularity for $n>4$.
\end{proof}

\section{Non-concentration cases: a Luckhaus type lemma for connections}\label{sec: Luckhaus}
We prove an annular interpolation estimate and use it to exclude curvature concentration.
\begin{lemma}[Luckhaus type lemma for connections]\label{Lemma: Luckhaus type lemma for connections}
Let $G$ be a compact matrix Lie group with Lie algebra $\g$, let $n\ge5$, $\rho>0$ and $\lambda\in(0,1/2)$. Let $A_1$ and $A_2$ be tangential connection forms of class $W^{1,2}\cap L^4$ on $\mathbb S_\rho$ and $\mathbb S_{(1-\lambda)\rho}$, respectively. Writing $\delta_s(\theta):=s\theta$, set
\begin{align*}
 a:=\delta_\rho^*A_1,\qquad b:=\delta_{(1-\lambda)\rho}^*A_2.
\end{align*}
Then there exists $A_\lambda\in(W^{1,2}\cap L^4)(\mathbb B_\rho\smallsetminus\mathbb B_{(1-\lambda)\rho})$ with the prescribed tangential boundary traces and
\begin{align}\label{Equation: corrected annular interpolation}
\int_{\mathbb B_\rho\smallsetminus\mathbb B_{(1-\lambda)\rho}}|F_{A_\lambda}|^2
&\le C\rho^{n-4}\lambda\left(\|F_a\|_{L^2}^2+\|F_b\|_{L^2}^2+\|(a-b)\wedge(a-b)\|_{L^2}^2\right)\nonumber\\
&\quad+C\rho^{n-4}\lambda^{-1}\|a-b\|_{L^2}^2,
\end{align}
where the norms on the right are taken on $\mathbb S^{n-1}$ and $C$ depends only on $n$ and $G$.
\begin{proof}
Let $\pi(x):=x/|x|$ and $t(r):=(\rho-r)/(\lambda\rho)$. Define the connection, with zero radial component, by
\begin{align*}
 A_\lambda(r,\theta):=(1-t(r))a(\theta)+t(r)b(\theta).
\end{align*}
Here and below tangential forms on $\mathbb S^{n-1}$ are identified with their pullbacks by $\pi$. Its curvature is
\begin{align*}
 F_{A_\lambda}=(1-t)F_a+tF_b-t(1-t)(a-b)\wedge(a-b)+t'\,dr\wedge(b-a).
\end{align*}
The first three terms are tangential and the last one is radial. Since $(1-\lambda)\rho\le r\le\rho$ and $\lambda<1/2$, integration in polar coordinates gives
\begin{align*}
 \int_{(1-\lambda)\rho}^{\rho}r^{n-5}\,dr\le C\lambda\rho^{n-4},\qquad
 \frac1{(\lambda\rho)^2}\int_{(1-\lambda)\rho}^{\rho}r^{n-3}\,dr\le C\lambda^{-1}\rho^{n-4}.
\end{align*}
These inequalities prove \eqref{Equation: corrected annular interpolation}. The asserted Sobolev regularity and the boundary identities follow directly from the definition.
\end{proof}
\end{lemma}
\begin{remark}
The radial term in \eqref{Equation: corrected annular interpolation} is sharp in its dependence on $\lambda$. For $G=U(1)$, $\rho=1$, $a=0$ and a nonzero co-closed form $b$, pairing the radial curvature with $b$ and applying Cauchy--Schwarz gives
\begin{align*}
 \int_{\mathbb B_1\smallsetminus\mathbb B_{1-\lambda}}|F_A|^2\ge c\lambda^{-1}\|b\|_{L^2(\mathbb S^{n-1})}^2
\end{align*}
for every interpolant $A$.
\end{remark}
\begin{lemma}[Non-concentration in a common gauge]\label{Lemma: no concentration}
Let $G$ be a compact matrix Lie group with Lie algebra $\g$, let $n\ge5$ and let $A$ be an almost $\operatorname{YM}$-energy minimizer on an open set $\Omega\subset\mathbb R^n$. Let $y\in\Omega$, $\rho_i\downarrow0$, and write $A_i:=\tau_{y,\rho_i}^*A$. Suppose that the curvatures converge, modulo gauge transformations, as
\begin{align*}
 F_{A_i}\rightharpoonup F_\varphi\quad\text{weakly in }L^2(\mathbb B_2),
\end{align*}
where $\varphi=\pi^*a_0$ and $a_0$ is a smooth connection form on $\mathbb S^{n-1}$. Set $U:=\mathbb B_{3/2}\smallsetminus\overline{\mathbb B_1}$. Assume that there are gauges $g_i\in W^{2,2}(U,G)\cap W^{1,4}(U,G)$ such that, writing $B_i:=A_i^{g_i}$,
\begin{align}\label{Equation: annular interpolation criterion}
 \|B_i-\varphi\|_{L^2(U)}^2\left(1+\|(B_i-\varphi)\wedge(B_i-\varphi)\|_{L^2(U)}^2\right)\rightarrow0.
\end{align}
Then $F_{A_i}\to F_\varphi$ strongly in $L^2(\mathbb B_1)$ modulo gauge transformations.
It is enough to assume almost minimality for balls centered at $y$, as in Remark \ref{Remark: centered almost minimality}.
\begin{proof}
Weak convergence gives $\sup_i\|F_{A_i}\|_{L^2(\mathbb B_2)}<\infty$. Put
\begin{align*}
 D_i&:=\int_U|B_i-\varphi|^2,\\
 M_i&:=1+\int_U\left(|F_{A_i}|^2+|F_\varphi|^2+|(B_i-\varphi)\wedge(B_i-\varphi)|^2\right).
\end{align*}
Thus $D_iM_i\to0$ by \eqref{Equation: annular interpolation criterion}. By the coarea formula and Chebyshev's inequality we can choose $r_i\in(1,3/2)$ such that the tangential traces of $B_i$ are of class $W^{1,2}\cap L^4$, the traces of $g_i$ are of class $W^{2,2}\cap W^{1,4}$, and, with $a_i:=\delta_{r_i}^*\iota_{\mathbb S_{r_i}}^*B_i$, we have simultaneously
\begin{align*}
 \|a_i-a_0\|_{L^2(\mathbb S^{n-1})}^2&\le C D_i,\\
 \|F_{a_i}\|_{L^2(\mathbb S^{n-1})}^2+\|F_{a_0}\|_{L^2(\mathbb S^{n-1})}^2+\|(a_i-a_0)\wedge(a_i-a_0)\|_{L^2(\mathbb S^{n-1})}^2&\le C M_i.
\end{align*}
The factors arising from radial pullback are uniformly bounded because $1<r_i<3/2$. Let
\begin{align*}
 \lambda_i:=\left(\frac{D_i+(i^2M_i)^{-1}}{M_i}\right)^{1/2}.
\end{align*}
After discarding finitely many terms, $0<\lambda_i<1/2$, $\lambda_i\to0$, and
\begin{align*}
 \lambda_iM_i+\lambda_i^{-1}D_i\le2\sqrt{D_iM_i+i^{-2}}\rightarrow0.
\end{align*}
Apply Lemma \ref{Lemma: Luckhaus type lemma for connections} with outer trace $\iota_{\mathbb S_{r_i}}^*B_i$ and inner trace $\iota_{\mathbb S_{(1-\lambda_i)r_i}}^*\varphi$. The resulting interpolant has energy $o(1)$. Gluing it to $\varphi$ on $\mathbb B_{(1-\lambda_i)r_i}$ gives $\widetilde B_i\in W^{1,2}\cap L^4(\mathbb B_{r_i})$. At the inner interface the full traces agree, since both connections have zero radial component. Extend $g_i|_{\mathbb S_{r_i}}$ radially to $\mathbb B_{r_i}$ and denote this extension by $\widetilde g_i$. Since $n>4$, it belongs to $W^{2,2}\cap W^{1,4}$. Thus $\widetilde A_i:=\widetilde B_i^{\widetilde g_i^{-1}}$ belongs to $W^{1,2}\cap L^4$ and has the same tangential boundary trace as $A_i$. Gauge invariance, almost minimality and homogeneity therefore give
\begin{align*}
 \int_{\mathbb B_{r_i}}|F_{A_i}|^2
 &\le\int_{\mathbb B_{r_i}}|F_{\widetilde A_i}|^2+C\rho_i^\alpha r_i^{n-4+\alpha}\\
 &\le(1-\lambda_i)^{n-4}\int_{\mathbb B_{r_i}}|F_\varphi|^2+o(1).
\end{align*}
Take any subsequence realizing the upper limit of $\int_{\mathbb B_1}|F_{A_i}|^2$ and pass to a further subsequence for which $r_i\to r\in[1,3/2]$. Weak lower semicontinuity on fixed subannuli implies
\begin{align*}
 \liminf_{i\to\infty}\int_{\mathbb B_{r_i}\smallsetminus\mathbb B_1}|F_{A_i}|^2
 \ge\int_{\mathbb B_r\smallsetminus\mathbb B_1}|F_\varphi|^2.
\end{align*}
Consequently,
\begin{align*}
 \limsup_{i\to\infty}\int_{\mathbb B_1}|F_{A_i}|^2
 \le\int_{\mathbb B_r}|F_\varphi|^2-\int_{\mathbb B_r\smallsetminus\mathbb B_1}|F_\varphi|^2
 =\int_{\mathbb B_1}|F_\varphi|^2.
\end{align*}
The reverse inequality follows from weak lower semicontinuity. Equality of the limiting norms proves strong $L^2$ convergence.
\end{proof}
\end{lemma}
\begin{remark}\label{Remark: sufficient annular assumptions}
Condition \eqref{Equation: annular interpolation criterion} holds if $B_i\to\varphi$ strongly in $L^2(U)$ and the wedge terms are bounded in $L^2(U)$, in particular if the $B_i$ are bounded in $L^4(U)$. For abelian differences the wedge term vanishes.

Weak convergence $B_i\rightharpoonup\varphi$ in $W^{1,2}(U)$ also suffices: Rellich compactness gives strong $L^2$ convergence, and
\begin{align*}
 (B_i-\varphi)\wedge(B_i-\varphi)=F_{B_i}-F_\varphi-d_\varphi(B_i-\varphi)
\end{align*}
is bounded in $L^2(U)$. More generally, strong $L^2$ convergence and a uniform $L^2$ bound on $d_\varphi(B_i-\varphi)$ suffice. A condition expressed solely in Lebesgue norms is
\begin{align*}
 \|B_i-\varphi\|_{L^2(U)}+\|B_i-\varphi\|_{L^2(U)}\|B_i-\varphi\|_{L^4(U)}^2\rightarrow0.
\end{align*}
\end{remark}
\begin{proposition}[Non-concentration for $\omega$-ASD connections]\label{Proposition: ASD nonconcentration}
Let $n\ge5$, let $\omega$ be smooth, and let $A\in(W^{1,2}\cap L^4)_{\loc}(\Omega,T^*\Omega\otimes\mathfrak g)$ be an $\omega$-ASD connection satisfying the regularity assumption of Lemma \ref{Lemma: torsion monotonicity}. Let $y\in\Omega$, let $\rho_i\downarrow0$, and suppose that, modulo gauge transformations,
\begin{align*}
F_{A_i}\rightharpoonup F_\varphi\quad\text{weakly in }L^2(\mathbb B_2),
\qquad A_i:=\tau_{y,\rho_i}^*A,
\end{align*}
where $\varphi=\pi^*a_0$ and $a_0$ is smooth on $\mathbb S^{n-1}$. Set $U:=\mathbb B_{3/2}\smallsetminus\overline{\mathbb B_1}$. Assume that there are gauges $g_i\in W^{2,2}(U,G)\cap W^{1,4}(U,G)$ such that
\begin{align}\label{Equation: ASD annular criterion}
\|a_i\|_{L^2(U)}+\|\operatorname{tr}(a_i\wedge a_i\wedge a_i)\|_{L^1(U)}\rightarrow0,
\qquad a_i:=A_i^{g_i}-\varphi.
\end{align}
Then $F_{A_i}\to F_\varphi$ strongly in $L^2(\mathbb B_1)$ modulo gauge transformations.
\begin{proof}
Write $B_i:=A_i^{g_i}$ and $a_i:=B_i-\varphi$. The relative Chern--Weil identity gives
\begin{align*}
\operatorname{tr}(F_{B_i}\wedge F_{B_i})-\operatorname{tr}(F_\varphi\wedge F_\varphi)=dQ_i,\qquad
Q_i:=\operatorname{tr}\left(a_i\wedge(F_{B_i}+F_\varphi)-\frac13a_i\wedge a_i\wedge a_i\right).
\end{align*}
This identity holds distributionally by strong approximation in $W^{1,2}\cap L^4$. Moreover,
\begin{align*}
\|Q_i\|_{L^1(U)}
\le C\|a_i\|_{L^2(U)}\big(\|F_{B_i}\|_{L^2(U)}+\|F_\varphi\|_{L^2(U)}\big)
+\frac13\|\operatorname{tr}(a_i\wedge a_i\wedge a_i)\|_{L^1(U)}\rightarrow0.
\end{align*}
Let $\omega_i:=\rho_i^{4-n}\tau_{y,\rho_i}^*\omega$, so that its coefficients converge smoothly to those of the constant form $\omega_y:=\omega(y)$. The connections $A_i$ and $B_i$ are $\omega_i$-ASD. Strong $L^2$ convergence of $B_i$ to $\varphi$ implies distributional convergence of their curvatures; the uniform curvature bound then gives weak $L^2$ convergence on $U$. Passing to the limit in the ASD equation shows that $\varphi$ is $\omega_y$-ASD. The preceding identity, tested against a compactly supported smooth function times $\omega_i$, gives
\begin{align*}
|F_{A_i}|^2\,d\mathcal L^n
\rightharpoonup |F_\varphi|^2\,d\mathcal L^n
\qquad\text{as measures on }U.
\end{align*}
Here the terms containing $Q_i$ tend to zero because $\omega_i$ is bounded in $C^1$.

By Lemma \ref{Lemma: torsion monotonicity}, the density
\begin{align*}
\Theta:=\lim_{r\downarrow0}r^{4-n}\int_{B_r(y)}|F_A|^2
\end{align*}
exists and is finite. Consequently, for every fixed $s\in(0,2)$,
\begin{align*}
\int_{\mathbb B_s}|F_{A_i}|^2\rightarrow\Theta s^{n-4}.
\end{align*}
Testing the annular measure convergence with a nonzero nonnegative smooth radial function supported in $U$, and using the homogeneity of $\varphi$, yields
\begin{align*}
\Theta=\int_{\mathbb B_1}|F_\varphi|^2.
\end{align*}
Thus the curvature norms on $\mathbb B_1$ converge. Together with the assumed weak convergence, this proves the conclusion.
\end{proof}
\end{proposition}
\begin{remark}\label{Remark: ASD sufficient conditions}
Condition \eqref{Equation: ASD annular criterion} follows from strong $L^3$ convergence of $A_i^{g_i}$ to $\varphi$ on $U$. It also follows from the annular interpolation condition, since
\begin{align*}
\|\operatorname{tr}(a_i\wedge a_i\wedge a_i)\|_{L^1(U)}
\le C\|a_i\|_{L^2(U)}\|a_i\wedge a_i\|_{L^2(U)}
\le C\|a_i\|_{L^2(U)}\|a_i\|_{L^4(U)}^2.
\end{align*}
By compact Sobolev embedding, weak $W^{1,p}(U)$ convergence suffices for $p>3n/(n+3)$, hence for $p>15/8$ in dimension $5$.
\end{remark}

\begin{lemma}[A curvature criterion for non-concentration]\label{Lemma: curvature criterion for nonconcentration}
Let $A$ be an almost $\operatorname{YM}$-energy minimizer on $\Omega\subset\mathbb R^n$, $n\ge5$, and let $y\in\Omega$ and $\rho_i\downarrow0$. Suppose that $F_{A_i}\rightharpoonup F_\varphi$ in $L^2(\mathbb B_2)$ modulo gauge, where $A_i:=\tau_{y,\rho_i}^*A$ and $\varphi=\pi^*a_0$ for a smooth connection form $a_0$ on $\mathbb S^{n-1}$. Let $U:=\mathbb B_{3/2}\smallsetminus\overline{\mathbb B_1}$. Suppose that $|F_{A_i}|\to|F_\varphi|$ almost everywhere on $U$ and that the functions $|F_{A_i}|^2$ are uniformly integrable on $U$. Then $F_{A_i}\to F_\varphi$ strongly in $L^2(\mathbb B_1)$ modulo gauge transformations.
\begin{proof}
By Vitali's convergence theorem,
\begin{align*}
 \int_U|F_{A_i}|^2\rightarrow\int_U|F_\varphi|^2.
\end{align*}
By Proposition \ref{Proposition: almost monotonicity formula}, the limit
\begin{align*}
 \Theta_A(y):=\lim_{r\downarrow0}r^{4-n}\int_{B_r(y)}|F_A|^2
\end{align*}
exists. Rescaling gives
\begin{align*}
 \lim_{i\to\infty}\int_U|F_{A_i}|^2
 =\left((3/2)^{n-4}-1\right)\Theta_A(y).
\end{align*}
On the other hand, homogeneity gives
\begin{align*}
 \int_U|F_\varphi|^2
 =\left((3/2)^{n-4}-1\right)\int_{\mathbb B_1}|F_\varphi|^2.
\end{align*}
Consequently $\Theta_A(y)=\int_{\mathbb B_1}|F_\varphi|^2$. Since $\int_{\mathbb B_1}|F_{A_i}|^2\to\Theta_A(y)$, weak convergence and convergence of the norms prove the result.
\end{proof}
\end{lemma}
\begin{remark}\label{Remark: higher integrability criterion}
The almost everywhere convergence in Lemma \ref{Lemma: curvature criterion for nonconcentration} holds for a tangent obtained by smooth convergence modulo gauge away from a concentration set of Lebesgue measure zero. By H\"older's inequality, uniform integrability follows from
\begin{align}\label{Equation: scale invariant higher integrability}
 \sup_i\int_U|F_{A_i}|^{2+\varepsilon}<\infty
 \qquad\text{for some }\varepsilon>0,
\end{align}
or, equivalently,
\begin{align*}
 \sup_i\rho_i^{4+2\varepsilon-n}\int_{B_{(3/2)\rho_i}(y)\smallsetminus B_{\rho_i}(y)}|F_A|^{2+\varepsilon}<\infty.
\end{align*}
\end{remark}
\begin{proof}[Proof of Corollary \ref{Theorem: higher dimensions}]
For $\omega$-ASD connections, centered almost minimality follows from Remark \ref{Remark: centered almost minimality}. Under condition (i), Lemma \ref{Lemma: no concentration} gives strong curvature convergence on $\mathbb B_1$; under condition (ii), Lemma \ref{Lemma: curvature criterion for nonconcentration} gives the same conclusion; under condition (iii), Proposition \ref{Proposition: ASD nonconcentration} applies. Replacing $\rho_i$ by $\rho_i/2$ gives strong convergence on $\mathbb B_2$. Theorem \ref{Theorem: uniqueness of tangent connections with isolated singularities} now applies. The sufficient conditions follow from Remarks \ref{Remark: sufficient annular assumptions} and \ref{Remark: ASD sufficient conditions}, H\"older's inequality and the pointwise estimate $|a\wedge a|\le C_G|a|^2$.
\end{proof}

\begin{remark}[Hermitian Yang--Mills examples]\label{Remark: concrete HYM examples}
In complex dimension three, consider a reflexive sheaf with a local presentation
\begin{align*}
0\rightarrow\mathcal O\xrightarrow{(f_1,f_2,f_3)}\mathcal O^{\oplus3}\rightarrow\mathcal E\rightarrow0,
\qquad f_j=P_j+O(|z|^{d+1}),
\end{align*}
where the $P_j$ are homogeneous of degree $d\ge1$ and have no common zero in $\mathbb{CP}^2$. This is a nonempty Zariski-open condition on their coefficients. By \cite[Corollary 3.37]{ChenSunDuke} and \cite[Theorem 1.1]{ChenSunInvent}, the algebraic tangent cone is a stable vector bundle, the analytic blow-up has no bubbling, and the spherical link is smooth. Thus the annular hypotheses hold for admissible Hermitian Yang--Mills connections in this class.

Nonconical examples occur on the stable reflexive sheaves over $\mathbb{CP}^3$ constructed in \cite[Section 3.4]{ChenSunDuke}. Choose a K\"ahler metric in the Fubini--Study class flat near the singularity and apply \cite{BandoSiu}. Subtracting a smooth central connection locally normalizes the Einstein constant. Adding a flat line gives an $\omega$-ASD connection on a trivial bundle, since the displayed sequence splits smoothly away from $0$. The smooth annular asymptotics and the asymptotic gauge of \cite{Yang} give $|A|=O(r^{-1})$ and $|\nabla A|=O(r^{-2})$, hence a $W^{1,2}\cap L^4$ representative. Tangent uniqueness for these examples follows from the cited results.
\end{remark}

\begin{remark}[Examples with non-closed forms]\label{Remark: conformal examples}
Let $A_0$ be a nonflat homogeneous $\omega_0$-ASD connection on a trivial bundle, with smooth link and constant $\omega_0$ satisfying \eqref{Equation: algebraic bound for omega}. One example is the primitive HYM cone associated with $T\mathbb{CP}^2\oplus\mathcal O(1)$; see \cite[Proposition 1.4 and Example 1.6]{JacobEnriqueWalpuski}. Its spherical pullback is smoothly trivial by the Euler sequence.
For small $b\in\mathbb R^n$, set
\begin{align*}
 f_b(x)&:=c_b(x)(x-|x|^2b),
 &c_b(x)&:=(1-2\langle b,x\rangle+|b|^2|x|^2)^{-1},\\
 A&:=f_b^*A_0,
 &\omega&:=c_b^{4-n}f_b^*\omega_0.
\end{align*}
Since $f_b^*g=c_b^2g$, the connection $A$ is $\omega$-ASD and the algebraic bound is preserved. Moreover,
\begin{align*}
 d_A(*_gF_A)=-(n-4)d\log c_b\wedge *_gF_A.
\end{align*}
Its linearization at $b=0$ is $-2(n-4)b^\flat\wedge *_gF_{A_0}$, which is nonzero for a suitable $b$. Thus $A$ is not Yang--Mills for $g$, and $d\omega\ne0$. The bounds $|A|=O(r^{-1})$ and $|\nabla A|=O(r^{-2})$ give $A\in W^{1,2}\cap L^4$. With $f_{b,r}(x):=r^{-1}f_b(rx)$, homogeneity gives
\begin{align*}
 \tau_{0,r}^*A=f_{b,r}^*A_0\rightarrow A_0
 \quad\text{smoothly on compact annuli}.
\end{align*}
The annular hypotheses therefore hold. These connections are Yang--Mills for the smooth metric $f_b^*g$.
\end{remark}

\subsection{Almost minimizing examples}
\begin{remark}\label{Remark: five dimensional examples}
Let $I$ be a smooth anti-self-dual $SU(2)$ connection on a bundle $P\rightarrow\mathbb S^4$, for example a classical instanton of any positive charge \cite{donaldson-kronheimer}. Its radial extension $A_0:=\pi^*I$ minimizes the energy on the fixed radial bundle: it is $dr$-ASD, and $dr$ is closed. For a comparison ball $D$, apply the relative Chern--Weil identity on $D\smallsetminus\overline{B_\varepsilon(0)}$; the inner boundary term vanishes because the pullback of $dr$ is zero. Strong approximation in $W^{1,2}\cap L^4$ with fixed outer trace and passage to the limit $\varepsilon\downarrow0$ give the comparison for Sobolev competitors. Here $dr$ is smooth on the punctured ball.

If $b\in(W^{1,2}\cap L^4)(B_1,\Lambda^1\mathbb R^5)$ satisfies
\begin{align}\label{Equation: abelian Morrey perturbation}
 \int_{B_s(x)}|db|^2\le Cs^{1+\alpha},\qquad \alpha>0,
\end{align}
then $A:=(A_0,ib)$ is almost minimizing with exponent $\alpha$, by orthogonality of the two factors. Moreover,
\begin{align*}
 \|\tau_{0,r}^*db\|_{L^2(B_2)}^2
 =r^{-1}\int_{B_{2r}}|db|^2\le Cr^\alpha.
\end{align*}
The Hodge estimate in abelian Coulomb gauges on fixed annuli gives strong $L^2$ convergence of the connection forms as well. Thus $(A_0,0)$ is a strong tangent, as required by Theorem \ref{Theorem: uniqueness for almost minimizers}-(i).
\end{remark}

\begin{example}[Radial curvature spikes]\label{Example: radial shell perturbation}
Let $n\ge5$, and let $\Phi:=\pi^*\phi$ be a minimizing cone with smooth link on a radial principal $G$-bundle, in the comparison class under consideration. Choose a nonzero smooth real one-form $\eta$ on $\mathbb S^{n-1}$ and a nonzero $\chi\in C_c^\infty((-1,1))$. Set
\begin{align*}
 r_k&:=2^{-k},\qquad \delta_k:=r_k^4/100,\\
 f(r)&:=\sum_{k\ge1}r_k^2\chi\!\left(\frac{r-r_k}{\delta_k}\right),
 \qquad b:=f(r)\pi^*\eta.
\end{align*}
On the product bundle with structure group $G\times U(1)$, the connection $A:=(\Phi,ib)$ is almost minimizing with exponent $1$ in the same comparison class, satisfies Theorem \ref{Theorem: uniqueness for almost minimizers}-(ii) for every finite $p>(n-1)/2$, and has tangent $(\Phi,0)$. Nevertheless,
\begin{align*}
 \limsup_{x\rightarrow0}|x|^2|F_A(x)|=\infty.
\end{align*}
For $n=5$, it also satisfies Theorem \ref{Theorem: uniqueness for almost minimizers}-(i).
\end{example}
\begin{proof}
The shells $S_k:=\{x:\lvert |x|-r_k\rvert<\delta_k\}$ are disjoint. On $S_k$,
\begin{align*}
 |b|\le Cr_k,\qquad |\nabla b|+|db|\le Cr_k^{-3},\qquad
 \int_{S_k}(|\nabla b|^2+|db|^2)\le Cr_k^{n-3}.
\end{align*}
Summability and $|b(x)|\le C|x|$ give $b\in W^{1,2}(B_1)\cap L^4(B_1)$. We also have
\begin{align}\label{Equation: radial shell Morrey estimate}
 \int_{B_s(x)\cap B_1}|db|^2\le Cs^{n-3}.
\end{align}
Indeed, if $|x|\le4s$, sum the shell estimates over $r_k\le Cs$. Otherwise, only a bounded number of shells intersect the ball, their radii are comparable to $|x|$, and the spherical cap estimate gives
\begin{align*}
 \int_{S_k\cap B_s(x)}|db|^2
 \le Cr_k^{-6}\delta_k s^{n-1}
 \le Cr_k^{-2}s^{n-1}\le Cs^{n-3}.
\end{align*}
On each ball allowed by the chosen comparison class, orthogonality of the factors and minimality of $\Phi$ give, for every competitor $B$ with the same boundary trace,
\begin{align*}
 \operatorname{YM}_{B_s(x)}(A)
 \le\operatorname{YM}_{B_s(x)}(B)+Cs^{n-4+1}.
\end{align*}
The rescaled spherical traces satisfy
\begin{align*}
 a(r)=(\phi,if(r)\eta),\qquad F_{a(r)}=(F_\phi,if(r)d\eta),
 \qquad |f(r)|\le Cr^2,
\end{align*}
which gives the tangential curvature bound. Choose $t_0$ and $\vartheta_0$ with $\chi'(t_0)\ne0$ and $\eta(\vartheta_0)\ne0$. At $x_k:=(r_k+\delta_k t_0)\vartheta_0$, the radial curvature gives
\begin{align*}
 |x_k|^2|F_A(x_k)|\ge c r_k\frac{r_k^2}{\delta_k}
 =c' r_k^{-1}\rightarrow\infty.
\end{align*}
This divergence persists for every smooth metric extending across $0$, by local equivalence of norms and distances. Also, $d^*db\ne0$: otherwise testing against a compactly supported shell summand would force its curvature to vanish. Finally, \eqref{Equation: radial shell Morrey estimate} gives
\begin{align*}
 r^{4-n}\int_{B_r}|db|^2\le Cr.
\end{align*}
Together with the smooth spherical convergence of $a(r)$, this proves strong annular convergence to $(\Phi,0)$.
\end{proof}

\begin{remark}[Nonflat examples]\label{Remark: all dimensional nonflat examples}
The flat cone gives examples in every $n\ge5$, and the instanton cones above give nonflat examples in dimension $5$. In dimension $6$, use the primitive HYM cone from Remark \ref{Remark: conformal examples}, with the Hopf contact correction on each summand as in \cite[Proposition 1.4 and Section 2]{JacobEnriqueWalpuski}. Its spherical bundle is trivial, its curvature is nonzero, and its radial coefficients give a global $W^{1,2}\cap L^4$ representative. The constant K\"ahler form and the relative Chern--Weil identity give minimality among global $W^{1,2}\cap L^4$ competitors with the same boundary trace; here we use the trivial-bundle variant of Theorem \ref{Theorem: uniqueness for almost minimizers}-(ii). The tangential curvature hypothesis also holds for the smoothly asymptotically conical connections discussed above.
\end{remark}

\appendix
\section{A criterion for the uniqueness of tangent cones} \label{sec: appendix criterion}
We record a decay criterion that applies to the logarithmic rates obtained from the log-epiperimetric inequality.
\begin{proposition}\label{Proposition: Dini decrease implies uniqueness of tangent cone}
    Let $G$ be a compact matrix Lie group with Lie algebra $\mathfrak{g}$, let $n\ge5$, and let $\Omega\subset\mathbb{R}^n$ be open. Let $y\in\Omega$ and let $A\in(W^{1,2}\cap L^4)(\Omega,T^*\Omega\otimes\mathfrak{g})$ satisfy the almost $\operatorname{YM}$-minimality comparison on balls centered at $y$. Assume that there exist $\rho_0\in(0,\operatorname{dist}(y,\partial\Omega))$ and a non-decreasing function $\phi:(0,\rho_0)\to(0,+\infty)$ such that
    \begin{align}\label{Equation: assumptions bound}
        \Theta(\rho,y;A)-\Theta(y;A)\le\phi(\rho)
        \qquad\forall\,\rho\in(0,\rho_0)
    \end{align}
    and
    \begin{align}\label{Equation: integrability condition around the origin}
        \int_0^{\rho_0}\frac{\sqrt{\phi(\rho)}}{\rho\sqrt{1+\log(\rho_0/\rho)}}\,d\rho<+\infty.
    \end{align}
    Then, in radial gauge about $y$, the rescaled spherical traces have a strong $L^2$ limit. In particular, any two tangent connections to $A$ at $y$ are gauge equivalent. The same conclusion holds for connections on smooth bundles over a punctured neighbourhood of $y$, with the local Sobolev regularity, finite energy and centered comparison class of \eqref{Equation: centered comparison all dimensions}, under the same decay assumptions.
\end{proposition}
\begin{proof}
    Decreasing $\rho_0$ if necessary, choose the weak radial gauge constructed in Subsection \ref{subsec: uniqueness almost minimizers five dimensions} and put $a(\rho):=\iota_{\mathbb{S}^{n-1}}^*A_{y,\rho}$ and $\nu_y(x):=(x-y)/|x-y|$. The map $a$ is locally absolutely continuous into $L^2$ and satisfies
    \begin{align}\label{Equation: appendix 2}
        \partial_\rho a(\rho)
        =\rho\,(F_A\res\nu_y)(\rho\,\cdot\,+y).
    \end{align}
    On a radial bundle the angular transition functions are independent of $\rho$, so the same construction applies and differences of traces are global adjoint-valued forms.
    By the almost monotonicity formula for centered comparisons, there exist $K,\alpha>0$ such that $E(\rho):=\Theta(\rho,y;A)-\Theta(y;A)+K\rho^\alpha$ is nonnegative and non-decreasing and, for $0<\sigma<\rho<\rho_0$,
    \begin{align*}
        \int_\sigma^\rho t\|a'(t)\|_{L^2(\mathbb{S}^{n-1})}^2\,dt
        \le C\big(E(\rho)-E(\sigma)\big)
        \le C\big(\phi(\rho)+K\rho^\alpha\big).
    \end{align*}
    Set $\Psi(\rho):=\phi(\rho)+K\rho^\alpha$. It is non-decreasing and satisfies \eqref{Equation: integrability condition around the origin}. Cauchy--Schwarz gives
    \begin{align*}
        \int_\sigma^\rho\|a'(t)\|_{L^2}\,dt
        \le C\big(\log(\rho/\sigma)\,\Psi(\rho)\big)^{1/2}.
    \end{align*}
    Put $\rho_j:=\rho_0e^{-2^j}$ for $j\ge0$. For $j\ge1$, the preceding estimate and the monotonicity of $\Psi$ give
    \begin{align*}
        \int_{\rho_{j+1}}^{\rho_j}\|a'(t)\|_{L^2}\,dt
        &\le C\sqrt{2^j\Psi(\rho_j)}\\
        &\le C\int_{\rho_j}^{\rho_{j-1}}
        \frac{\sqrt{\Psi(t)}}{t\sqrt{1+\log(\rho_0/t)}}\,dt.
    \end{align*}
    Indeed, the integral of the scalar weight on $(\rho_j,\rho_{j-1})$ is
    $2(\sqrt{1+2^j}-\sqrt{1+2^{j-1}})\ge c2^{j/2}$.
    Summing over $j\ge k$ yields
    \begin{align}\label{Equation: estimate uniqueness of blow ups}
        \int_0^{\rho_k}\|a'(t)\|_{L^2}\,dt
        \le C\int_0^{\rho_{k-1}}
        \frac{\sqrt{\Psi(t)}}{t\sqrt{1+\log(\rho_0/t)}}\,dt
        \rightarrow0
        \qquad\text{as }k\rightarrow+\infty.
    \end{align}
    Hence $a(\rho)$ has a unique strong $L^2$ limit as $\rho\downarrow0$. The rescaled radial representatives converge on every compact annulus to the homogeneous extension of this limit. On a region where another tangent is represented, the transition gauges satisfy the same first-order equation used in Section \ref{sec: proof of uniqueness}. Their local $W^{1,2}$ compactness identifies that tangent with the radial limit. Thus every tangent connection is gauge equivalent to the same conical connection.
\end{proof}
Condition \eqref{Equation: integrability condition around the origin} holds for $\phi(\rho)=C(1+\log(\rho_0/\rho))^{-1/\gamma}$ whenever $0<\gamma<1$, and for $\phi(\rho)=C\rho^\beta$ whenever $\beta>0$.

\bibliographystyle{amsalpha} 
\bibliography{main} 

\providecommand{\bysame}{\leavevmode\hbox to3em{\hrulefill}\thinspace}
\providecommand{\MR}{\relax\ifhmode\unskip\space\fi MR }
\providecommand{\MRhref}[2]{%
  \href{http://www.ams.org/mathscinet-getitem?mr=#1}{#2}
}
\providecommand{\href}[2]{#2}
\begin{thebibliography}{GPSVG16}

\bibitem[AA81]{AllardAlmgren}
William~K. Allard and Frederick~J. Almgren, Jr., \emph{On the radial behavior of minimal surfaces and the uniqueness of their tangent cones}, Ann. of Math. (2) \textbf{113} (1981), no.~2, 215--265. \MR{607893}

\bibitem[AS88]{AdamsSimon}
David Adams and Leon Simon, \emph{Rates of asymptotic convergence near isolated singularities of geometric extrema}, Indiana Univ. Math. J. \textbf{37} (1988), no.~2, 225--254. \MR{963501}

\bibitem[Bel14]{bellettini-p}
Costante Bellettini, \emph{Uniqueness of tangent cones to positive-{$(p,p)$} integral cycles}, Duke Math. J. \textbf{163} (2014), no.~4, 705--732. \MR{3178430}

\bibitem[BET20]{Badger}
Matthew Badger, Max Engelstein, and Tatiana Toro, \emph{Regularity of the singular set in a two-phase problem for harmonic measure with {H}\"older data}, Rev. Mat. Iberoam. \textbf{36} (2020), no.~5, 1375--1408. \MR{4161290}

\bibitem[BS94]{BandoSiu}
Shigetoshi Bando and Yum-Tong Siu, \emph{Stable sheaves and {E}instein-{H}ermitian metrics}, Geometry and analysis on complex manifolds, World Sci. Publ., River Edge, NJ, 1994, pp.~39--50. \MR{1463962}

\bibitem[CM14]{ColdingMinicozziInvent}
Tobias~Holck Colding and William~P. Minicozzi, II, \emph{On uniqueness of tangent cones for {E}instein manifolds}, Invent. Math. \textbf{196} (2014), no.~3, 515--588. \MR{3211041}

\bibitem[CM15]{ColdingMinicozzi}
\bysame, \emph{Uniqueness of blowups and Łojasiewicz inequalities}, Ann. of Math. (2) \textbf{182} (2015), no.~1, 221--285. \MR{3374960}

\bibitem[CR23]{caniato-riviere}
Riccardo Caniato and Tristan Rivi{\`e}re, \emph{The unique tangent cone property for weakly holomorphic maps into projective algebraic varieties}, Duke Mathematical Journal \textbf{172} (2023), no.~13, 2471--2536.

\bibitem[CR24]{caniato-riviere-ym}
Riccardo Caniato and Tristan Rivière, \emph{Coulomb gauges and regularity for stationary weak yang$-$mills connections in supercritical dimension}, 2024.

\bibitem[CS20a]{ChenSunIMRN}
Xuemiao Chen and Song Sun, \emph{Algebraic tangent cones of reflexive sheaves}, Int. Math. Res. Not. IMRN (2020), no.~24, 10042--10063. \MR{4190396}

\bibitem[CS20b]{ChenSunDuke}
\bysame, \emph{Singularities of {H}ermitian-{Y}ang-{M}ills connections and {H}arder-{N}arasimhan-{S}eshadri filtrations}, Duke Math. J. \textbf{169} (2020), no.~14, 2629--2695. \MR{4149506}

\bibitem[CS21a]{ChenSunGT}
\bysame, \emph{Analytic tangent cones of admissible {H}ermitian {Y}ang-{M}ills connections}, Geom. Topol. \textbf{25} (2021), no.~4, 2061--2108. \MR{4286369}

\bibitem[CS21b]{ChenSunInvent}
\bysame, \emph{Reflexive sheaves, {H}ermitian-{Y}ang-{M}ills connections, and tangent cones}, Invent. Math. \textbf{225} (2021), no.~1, 73--129. \MR{4270664}

\bibitem[CSV18]{ColomboSpolaorVelichkov1}
Maria Colombo, Luca Spolaor, and Bozhidar Velichkov, \emph{A logarithmic epiperimetric inequality for the obstacle problem}, Geom. Funct. Anal. \textbf{28} (2018), no.~4, 1029--1061. \MR{3820438}

\bibitem[CSV20a]{ColomboSpolaorVelichkov2}
\bysame, \emph{Direct epiperimetric inequalities for the thin obstacle problem and applications}, Comm. Pure Appl. Math. \textbf{73} (2020), no.~2, 384--420. \MR{4054360}

\bibitem[CSV20b]{colombo-spolaor-velichkov-3}
\bysame, \emph{On the asymptotic behavior of the solutions to parabolic variational inequalities}, J. Reine Angew. Math. \textbf{768} (2020), 149--182. \MR{4168689}

\bibitem[CW22]{wentworth-chen}
Xuemiao Chen and Richard~A. Wentworth, \emph{Compactness for {$\Omega$}-{Y}ang-{M}ills connections}, Calc. Var. Partial Differential Equations \textbf{61} (2022), no.~2, Paper No. 58, 30. \MR{4376537}

\bibitem[DK90]{donaldson-kronheimer}
S.~K. Donaldson and P.~B. Kronheimer, \emph{The geometry of four-manifolds}, Oxford Mathematical Monographs, The Clarendon Press, Oxford University Press, New York, 1990, Oxford Science Publications. \MR{1079726}

\bibitem[DL22]{de2022regularity}
Camillo De~Lellis, \emph{The regularity theory for the area functional (in geometric measure theory)}, International Congress of Mathematicians, 2022.

\bibitem[DLS16]{CenterManifold}
Camillo De~Lellis and Emanuele Spadaro, \emph{Regularity of area minimizing currents {II}: center manifold}, Ann. of Math. (2) \textbf{183} (2016), no.~2, 499--575. \MR{3450482}

\bibitem[DLSS17a]{BranchedCenterManifold}
Camillo De~Lellis, Emanuele Spadaro, and Luca Spolaor, \emph{Regularity theory for 2-dimensional almost minimal currents {II}: {B}ranched center manifold}, Ann. PDE \textbf{3} (2017), no.~2, Paper No. 18, 85. \MR{3712561}

\bibitem[DLSS17b]{DeLellisSpadaroSpolaor}
\bysame, \emph{Uniqueness of tangent cones for two-dimensional almost-minimizing currents}, Comm. Pure Appl. Math. \textbf{70} (2017), no.~7, 1402--1421. \MR{3666570}

\bibitem[Don83]{Donaldson}
S.~K. Donaldson, \emph{An application of gauge theory to four-dimensional topology}, J. Differential Geom. \textbf{18} (1983), no.~2, 279--315. \MR{710056}

\bibitem[DT98]{DonaldsonThomas}
S.~K. Donaldson and R.~P. Thomas, \emph{Gauge theory in higher dimensions}, The geometric universe ({O}xford, 1996), Oxford Univ. Press, Oxford, 1998, pp.~31--47. \MR{1634503}

\bibitem[ESV19]{EngelsteinSpolaorVelichkov}
Max Engelstein, Luca Spolaor, and Bozhidar Velichkov, \emph{({L}og-)epiperimetric inequality and regularity over smooth cones for almost area-minimizing currents}, Geom. Topol. \textbf{23} (2019), no.~1, 513--540. \MR{3921325}

\bibitem[ESV20]{ESVduke}
\bysame, \emph{Uniqueness of the blowup at isolated singularities for the {A}lt-{C}affarelli functional}, Duke Math. J. \textbf{169} (2020), no.~8, 1541--1601. \MR{4101738}

\bibitem[ESV24]{SymmetricEdelenSpolaorVelichkov}
Nick Edelen, Luca Spolaor, and Bozhidar Velichkov, \emph{The symmetric (log-)epiperimetric inequality and a decay-growth estimate}, Calc. Var. Partial Differential Equations \textbf{63} (2024), no.~1, Paper No. 2, 29. \MR{4668988}

\bibitem[Fee16]{FeehanMemoir}
Paul M.~N. Feehan, \emph{Global existence and convergence of solutions to gradient systems and applications to yang-mills gradient flow}, 2016.

\bibitem[Fee22]{Feehan}
Paul M.~N. Feehan, \emph{Optimal Łojasiewicz-{S}imon inequalities and {M}orse-{B}ott {Y}ang-{M}ills energy functions}, Adv. Calc. Var. \textbf{15} (2022), no.~4, 635--671. \MR{4489597}

\bibitem[FM20a]{FeehanMaridakisCrelle}
Paul M.~N. Feehan and Manousos Maridakis, \emph{Łojasiewicz-{S}imon gradient inequalities for analytic and {M}orse-{B}ott functions on {B}anach spaces}, J. Reine Angew. Math. \textbf{765} (2020), 35--67. \MR{4129355}

\bibitem[FM20b]{FeehanMaridakis}
\bysame, \emph{Łojasiewicz-{S}imon gradient inequalities for coupled {Y}ang-{M}ills energy functionals}, Mem. Amer. Math. Soc. \textbf{267} (2020), no.~1302, xiii+138. \MR{4199212}

\bibitem[Fre82]{Freedman}
Michael~Hartley Freedman, \emph{The topology of four-dimensional manifolds}, J. Differential Geometry \textbf{17} (1982), no.~3, 357--453. \MR{679066}

\bibitem[FS16]{FocardiSpadaro}
Matteo Focardi and Emanuele Spadaro, \emph{An epiperimetric inequality for the thin obstacle problem}, Adv. Differential Equations \textbf{21} (2016), no.~1-2, 153--200. \MR{3449333}

\bibitem[FU91]{freed-uhlenbeck}
Daniel~S. Freed and Karen~K. Uhlenbeck, \emph{Instantons and four-manifolds}, second ed., Mathematical Sciences Research Institute Publications, vol.~1, Springer-Verlag, New York, 1991. \MR{1081321}

\bibitem[Gom83]{Gompf}
Robert~E. Gompf, \emph{Three exotic {${\bf R}^{4}$}'s and other anomalies}, J. Differential Geom. \textbf{18} (1983), no.~2, 317--328. \MR{710057}

\bibitem[GPSVG16]{GarofaloPetrosyanMariana}
Nicola Garofalo, Arshak Petrosyan, and Mariana Smit Vega~Garcia, \emph{An epiperimetric inequality approach to the regularity of the free boundary in the {S}ignorini problem with variable coefficients}, J. Math. Pures Appl. (9) \textbf{105} (2016), no.~6, 745--787. \MR{3491531}

\bibitem[GW89]{GulliverWhite}
Robert Gulliver and Brian White, \emph{The rate of convergence of a harmonic map at a singular point}, Math. Ann. \textbf{283} (1989), no.~4, 539--549. \MR{990588}

\bibitem[HL87]{HardtLin}
Robert Hardt and Fang-Hua Lin, \emph{Mappings minimizing the {$L^p$} norm of the gradient}, Comm. Pure Appl. Math. \textbf{40} (1987), no.~5, 555--588. \MR{896767}

\bibitem[JSEW18]{JacobEnriqueWalpuski}
Adam Jacob, Henrique S\'a{}~Earp, and Thomas Walpuski, \emph{Tangent cones of {H}ermitian {Y}ang-{M}ills connections with isolated singularities}, Math. Res. Lett. \textbf{25} (2018), no.~5, 1429--1445. \MR{3917734}

\bibitem[JW18]{JacobWalpuski}
Adam Jacob and Thomas Walpuski, \emph{Hermitian {Y}ang-{M}ills metrics on reflexive sheaves over asymptotically cylindrical {K}\"ahler manifolds}, Comm. Partial Differential Equations \textbf{43} (2018), no.~11, 1566--1598. \MR{3924216}

\bibitem[Lin99]{lin-gradient-estimates}
Fang-Hua Lin, \emph{Gradient estimates and blow-up analysis for stationary harmonic maps}, Ann. of Math. (2) \textbf{149} (1999), no.~3, 785--829. \MR{1709303}

\bibitem[Luc88]{Luckhaus}
Stephan Luckhaus, \emph{Partial {H}\"older continuity for minima of certain energies among maps into a {R}iemannian manifold}, Indiana Univ. Math. J. \textbf{37} (1988), no.~2, 349--367. \MR{963506}

\bibitem[MMR94]{MMR}
John~W. Morgan, Tomasz Mrowka, and Daniel Ruberman, \emph{The {$L^2$}-moduli space and a vanishing theorem for {D}onaldson polynomial invariants}, Monographs in Geometry and Topology, vol.~II, International Press, Cambridge, MA, 1994. \MR{1287851}

\bibitem[MR03]{meyer-riviere}
Yves Meyer and Tristan Rivi\`ere, \emph{A partial regularity result for a class of stationary {Y}ang-{M}ills fields in high dimension}, Rev. Mat. Iberoamericana \textbf{19} (2003), no.~1, 195--219. \MR{1993420}

\bibitem[Nak88]{Nakajima}
Hiraku Nakajima, \emph{Compactness of the moduli space of {Y}ang-{M}ills connections in higher dimensions}, J. Math. Soc. Japan \textbf{40} (1988), no.~3, 383--392. \MR{945342}

\bibitem[PR10]{pumberger-riviere}
David Pumberger and Tristan Rivi\`ere, \emph{Uniqueness of tangent cones for semicalibrated integral 2-cycles}, Duke Math. J. \textbf{152} (2010), no.~3, 441--480. \MR{2654220}

\bibitem[PR17]{petrache-riviere}
Mircea Petrache and Tristan Rivi\`ere, \emph{The resolution of the {Y}ang-{M}ills {P}lateau problem in super-critical dimensions}, Adv. Math. \textbf{316} (2017), 469--540. \MR{3672911}

\bibitem[Pri83]{price-monotonicity-formula}
Peter Price, \emph{A monotonicity formula for {Y}ang-{M}ills fields}, Manuscripta Math. \textbf{43} (1983), no.~2-3, 131--166. \MR{707042}

\bibitem[R{\aa}d92]{Rade}
Johan R{\aa}de, \emph{On the {Y}ang-{M}ills heat equation in two and three dimensions}, J. Reine Angew. Math. \textbf{431} (1992), 123--163. \MR{1179335}

\bibitem[Rei64a]{Reifenberg}
E.~R. Reifenberg, \emph{An epiperimetric inequality related to the analyticity of minimal surfaces}, Ann. of Math. (2) \textbf{80} (1964), 1--14. \MR{171197}

\bibitem[Rei64b]{ReifenbergAnalytic}
\bysame, \emph{On the analyticity of minimal surfaces}, Ann. of Math. (2) \textbf{80} (1964), 15--21. \MR{171198}

\bibitem[Riv04]{Riviere}
Tristan Rivi\`ere, \emph{A lower-epiperimetric inequality for area-minimizing surfaces}, Comm. Pure Appl. Math. \textbf{57} (2004), no.~12, 1673--1685. \MR{2082243}

\bibitem[Riv20]{riviere-variations}
\bysame, \emph{The variations of {Y}ang-{M}ills {L}agrangian}, Geometric analysis---in honor of {G}ang {T}ian's 60th birthday, Progr. Math., vol. 333, Birkh\"{a}user/Springer, Cham, [2020] \copyright 2020, pp.~305--379. \MR{4181007}

\bibitem[RT04]{riviere-tian-maps}
Tristan Rivi\`ere and Gang Tian, \emph{The singular set of {$J$}-holomorphic maps into projective algebraic varieties}, J. Reine Angew. Math. \textbf{570} (2004), 47--87. \MR{2075762}

\bibitem[RT09]{RiviereTian}
\bysame, \emph{The singular set of 1-1 integral currents}, Ann. of Math. (2) \textbf{169} (2009), no.~3, 741--794. \MR{2480617}

\bibitem[SEW15]{EarpWalpuski}
Henrique~N. S\'a{}~Earp and Thomas Walpuski, \emph{{$\rm {G}_2$}-instantons over twisted connected sums}, Geom. Topol. \textbf{19} (2015), no.~3, 1263--1285. \MR{3352236}

\bibitem[Sim83]{AsymptoticSimon}
Leon Simon, \emph{Asymptotics for a class of nonlinear evolution equations, with applications to geometric problems}, Ann. of Math. (2) \textbf{118} (1983), no.~3, 525--571. \MR{727703}

\bibitem[Sim93]{CylindricalJDG}
\bysame, \emph{Cylindrical tangent cones and the singular set of minimal submanifolds}, Journal of Differential Geometry \textbf{38} (1993), no.~3, 585--652.

\bibitem[Sim94]{CylindricalCAG}
\bysame, \emph{Uniqueness of some cylindrical tangent cones}, Communications in Analysis and Geometry \textbf{2} (1994), no.~1, 1--33.

\bibitem[Spo19]{Spolaor}
Luca Spolaor, \emph{Almgren's type regularity for semicalibrated currents}, Adv. Math. \textbf{350} (2019), 747--815. \MR{3948685}

\bibitem[SU82]{schoen-uhlenbeck}
Richard Schoen and Karen Uhlenbeck, \emph{A regularity theory for harmonic maps}, Journal of Differential Geometry \textbf{17} (1982), no.~2, 307--335.

\bibitem[SU22]{SmithUhlenbeck}
Penny Smith and Karen Uhlenbeck, \emph{Removeability of a codimension four singular set for solutions of a {Yang Mills Higgs} equation with small energy}, Surveys in differential geometry, 2019: {Differential} geometry, {Calabi--Yau} theory, and general relativity (Huai-Dong Cao and Shing-Tung Yau, eds.), Surveys in Differential Geometry, vol.~24, International Press, Boston, MA, 2022, pp.~257--291. \MR{4479723}

\bibitem[SV19]{SpolaorVelichkov}
Luca Spolaor and Bozhidar Velichkov, \emph{An epiperimetric inequality for the regularity of some free boundary problems: the 2-dimensional case}, Comm. Pure Appl. Math. \textbf{72} (2019), no.~2, 375--421. \MR{3896024}

\bibitem[SV21]{SpolaorVelichkovEng}
\bysame, \emph{On the logarithmic epiperimetric inequality for the obstacle problem}, Math. Eng. \textbf{3} (2021), no.~1, Paper No. 4, 42. \MR{4144099}

\bibitem[Sz{\'e}20]{Gabor}
G{\'a}bor Sz{\'e}kelyhidi, \emph{Uniqueness of certain cylindrical tangent cones}, 2020.

\bibitem[Tau87]{Taubes}
Clifford~Henry Taubes, \emph{Gauge theory on asymptotically periodic {$4$}-manifolds}, J. Differential Geom. \textbf{25} (1987), no.~3, 363--430. \MR{882829}

\bibitem[Tay73]{TaylorFlatChains}
Jean~E. Taylor, \emph{Regularity of the singular sets of two-dimensional area-minimizing flat chains modulo {$3$} in {$R\sp{3}$}}, Invent. Math. \textbf{22} (1973), 119--159. \MR{333903}

\bibitem[Tay76a]{TaylorSoapBubble}
\bysame, \emph{The structure of singularities in soap-bubble-like and soap-film-like minimal surfaces}, Ann. of Math. (2) \textbf{103} (1976), no.~3, 489--539. \MR{428181}

\bibitem[Tay76b]{TaylorEllipsoidal}
\bysame, \emph{The structure of singularities in solutions to ellipsoidal variational problems with constraints in {${\rm R}\sp{3}$}}, Ann. of Math. (2) \textbf{103} (1976), no.~3, 541--546. \MR{428182}

\bibitem[Tay77]{TaylorCapillarity}
\bysame, \emph{Boundary regularity for solutions to various capillarity and free boundary problems}, Comm. Partial Differential Equations \textbf{2} (1977), no.~4, 323--357. \MR{487721}

\bibitem[Tia00]{tian-gauge-theory}
Gang Tian, \emph{Gauge theory and calibrated geometry. {I}}, Ann. of Math. (2) \textbf{151} (2000), no.~1, 193--268. \MR{1745014}

\bibitem[TT04]{TaoTian}
Terence Tao and Gang Tian, \emph{A singularity removal theorem for {Y}ang-{M}ills fields in higher dimensions}, J. Amer. Math. Soc. \textbf{17} (2004), no.~3, 557--593. \MR{2053951}

\bibitem[Uhl82]{uhlenbeck-connections}
Karen~K. Uhlenbeck, \emph{Connections with {$L^{p}$} bounds on curvature}, Comm. Math. Phys. \textbf{83} (1982), no.~1, 31--42. \MR{648356}

\bibitem[Wal19]{Waldron}
Alex Waldron, \emph{Long-time existence for {Y}ang-{M}ills flow}, Invent. Math. \textbf{217} (2019), no.~3, 1069--1147. \MR{3989258}

\bibitem[Weh04]{wehrheim}
Katrin Wehrheim, \emph{Uhlenbeck compactness}, EMS Series of Lectures in Mathematics, European Mathematical Society (EMS), Z\"{u}rich, 2004. \MR{2030823}

\bibitem[Wei99]{Weiss}
Georg~S. Weiss, \emph{A homogeneity improvement approach to the obstacle problem}, Invent. Math. \textbf{138} (1999), no.~1, 23--50. \MR{1714335}

\bibitem[Whi83]{White}
Brian White, \emph{Tangent cones to two-dimensional area-minimizing integral currents are unique}, Duke Math. J. \textbf{50} (1983), no.~1, 143--160. \MR{700134}

\bibitem[Wic14]{wickramasekera2014regularity}
Neshan Wickramasekera, \emph{Regularity of stable minimal hypersurfaces: Recent advances in the theory and applications}, Surveys in differential geometry \textbf{19} (2014), no.~1, 231--301.

\bibitem[Yan03]{Yang}
Baozhong Yang, \emph{The uniqueness of tangent cones for {Y}ang-{M}ills connections with isolated singularities}, Adv. Math. \textbf{180} (2003), no.~2, 648--691. \MR{2020554}

\bibitem[Ło65]{Loj}
Stanislaw Łojasiewicz, \emph{Ensembles semi-analytiques}, IHES notes (1965), 220.

\end{thebibliography}
\end{document}